\documentclass[11pt]{article}
\usepackage{odonnell}

\newcommand{\hamlevel}{\kappa}
\newcommand{\hamwt}[2]{\#_{#1}#2}
\newcommand{\colors}{{\ell}}
\newcommand{\slice}[1]{\calU_{#1}}
\newcommand{\symm}[1]{S_{#1}}
\newcommand{\transpos}[1]{\mathrm{Trans}(#1)}

\newcommand{\Energy}{\calE}
\newcommand{\Ker}{\mathrm{K}}
\newcommand{\Lap}{\mathrm{L}}
\newcommand{\Heat}{\mathrm{H}}
\newcommand{\Ent}{\mathbf{Ent}}

\newcommand{\invLS}[1]{R_{#1}}

\newcommand{\lwr}{\mathbf{lower}}
\newcommand{\upr}{\mathbf{upper}}

\DeclarePairedDelimiterX\diverg[2]{(}{)}{#1 \mathrel{}\mathclose{}\delimsize\|\mathopen{}\mathrel{} #2}
\newcommand{\KL}[2]{\mathrm{KL}\diverg{#1}{#2}}

\newcommand{\ep}{\epsilon}

\newcommand{\lb}{\lambda}

\newcommand{\cost}{c}
\newcommand{\corr}{\mathrm{corr}}

\renewcommand{\bdry}{\partial}

\begin{document}

\title{A log-Sobolev inequality for the multislice, with applications}
\author{Yuval Filmus\thanks{Technion Computer Science Department. \texttt{yuvalfi@cs.technion.ac.il}. Taub Fellow --- supported by the Taub Foundations. The research was funded by ISF grant 1337/16.} \and Ryan O'Donnell\thanks{Computer Science Department, Carnegie Mellon University.  \texttt{odonnell@cs.cmu.edu}. Supported by NSF grant CCF-1717606. This material is based upon work supported by the National Science Foundation under grant numbers listed above. Any opinions, findings and conclusions or recommendations expressed in this material are those of the author and do not necessarily reflect the views of the National Science Foundation (NSF).} \and Xinyu Wu\thanks{Computer Science Department, Carnegie Mellon University.  \texttt{xinyuw1@andrew.cmu.edu}}}

\maketitle

\begin{abstract}
    Let $\hamlevel \in \N_+^\colors$ satisfy $\hamlevel_1 + \cdots + \hamlevel_\colors = n$, and let $\slice{\hamlevel}$ denote the \emph{multislice} of all strings $u \in [\colors]^n$ having exactly $\hamlevel_i$ coordinates equal to~$i$, for all $i \in [\colors]$.  Consider the Markov chain on $\slice{\hamlevel}$ where a step is a random transposition of two coordinates of~$u$.  We show that the log-Sobolev constant~$\varrho_\hamlevel$ for the chain satisfies
    \[
        \varrho_\hamlevel^{-1} \leq
        n \cdot \sum_{i=1}^\colors \tfrac{1}{2} \log_2(4n/\hamlevel_i),
    \]
     which is sharp up to constants whenever $\colors$ is constant.  From this, we derive some consequences for small-set expansion and isoperimetry in the multislice, including a KKL Theorem, a Kruskal--Katona Theorem for the multislice, a Friedgut Junta Theorem, and a Nisan--Szegedy Theorem.
\end{abstract}

\section{Introduction}

Suppose we have a deck of $n$ cards, with $\hamlevel_1$ of them colored red, $\hamlevel_2$ of them colored blue, and $\hamlevel_3$~of them colored green.  If we ``shuffle'' the cards by repeatedly transposing random pairs of cards, how long does it take for the deck to get to a well-mixed configuration?  This question is asking about the mixing time and expansion in a Markov chain known variously as the \emph{multi-urn Bernoulli--Laplace diffusion process} or the \emph{multislice}.

Let $\colors \in \N_+$ denote a number of \emph{colors} and let $n \in \N_+$ denote a number of \emph{coordinates} (or \emph{positions}). Following computer science terminology, we refer to elements $u \in [\colors]^n$ as \emph{strings}.  Given a color~$i \in [\colors]$, we write $\hamwt{i}{u}$ for the number of coordinates~$j \in [n]$ for which $u_j = i$. The vector $\hamlevel = (\hamwt{1}{u}, \dots, \hamwt{\colors}{u}) \in \N^\colors$ is referred to as the \emph{histogram} of~$u$. In general, if  $\hamlevel \in \N_+^\colors$ satisfies $\hamlevel_1 + \cdots + \hamlevel_\colors = n$ (so $\hamlevel$ is a \emph{composition} of~$n$), we define the associated \emph{multislice} to be
\[
    \slice{\hamlevel} = \braces*{u \in [\colors]^n : \hamwt{i}{u} = \hamlevel_i \text{ for all } i \in [\colors]}.
\]
The terminology here is inspired by the well-studied case when $\colors = 2$, in which case $\slice{\hamlevel}$ is a Hamming \emph{slice} of the Boolean cube.  We also remark that when $\colors = n$ and $\hamlevel = (1, 1, \dots, 1)$, the set $\slice{\hamlevel}$ is the set of all permutations of~$[n]$.

\paragraph{The random transposition Markov chain.} The symmetric group $\symm{n}$ acts on strings $u \in[\colors]^n$ in the natural way, by permuting coordinates: $(u^\sigma)_j = u_{\sigma(j)}$ for $\sigma \in \symm{n}$.  This action preserves each multislice $\slice{\hamlevel}$.  This paper is concerned with the Markov chain on $\slice{\hamlevel}$ generated by \emph{random transpositions}.  Let $\transpos{n} \subseteq \symm{n}$ denote the set of transpositions on~$n$ coordinates.  We will specifically be interested in the reversible, discrete-time Markov chain on state space~$\slice{\hamlevel}$ in which a step from $u \in \slice{\hamlevel}$ consists of moving to $u^\btau$, where $\btau \sim \transpos{n}$ is chosen uniformly at random.  (We always use \textbf{boldface} to denote random variables.)  One also has the associated \emph{Schreier graph}, with vertex set~$\slice{\hamlevel}$ and edges $\{u, u^\tau\}$ for all $u \in \slice{\hamlevel}$ and $\tau \in \transpos{n}$.  Since this graph is regular, it follows that the invariant distribution for the Markov chain is the uniform distribution on~$\slice{\hamlevel}$.  We will denote this distribution by $\pi_\hamlevel$, or just $\pi$ if $\hamlevel$ is clear from context.

\paragraph{Log-Sobolev inequalities.}  One of the most powerful ways to study mixing time and ``small-set expansion'' in Markov chains is through \emph{log-Sobolev} inequalities (see, e.g.,~\cite{Gro75,DS96}).  For a subset $A \subseteq \slice{\hamlevel}$, define its \emph{conductance} (or \emph{expansion}) to be
\[
    \Phi[A] = \Pr_{\substack{\bu \sim A \\ \btau \sim \transpos{n}}}[\bu^\btau \not \in A].
\]
Sets $A$ with small conductance are natural bottlenecks for mixing in the Markov chain.  An example when $\colors = 2$ and $\hamlevel = (n/2, n/2)$ is the ``dictator'' set $A = \{u : u_1 = 1\}$. It has expansion $\Phi[A] = \frac{1}{n-1}$, and indeed, if we start the random walk from a string~$u$ with $u_1 = 1$, it will take about $n/2$ steps on average before there's even a chance that $u_1$ will change from~$1$.

One feature of this example is that the set~$A$ is ``large''; its \emph{(fractional) volume},
\[
    \vol(A) = \abs{A}/\abs{\slice{\hamlevel}} = \Pr_{\bu \sim \pi}[\bu \in A],
\]
is bounded below by a constant.  The ``small-set expansion'' phenomenon~\cite{KKL88,LK99,RS10} (occurring most famously in the standard random walk on the Boolean cube~$\{0,1\}^n$) refers to the possibility that all ``small'' sets have high conductance.  Intuitively, if small-set expansion holds for a Markov chain, then a random walk with a deterministic starting point should mix rapidly in its early stages, with the possibility for slowdown occurring only when the chain is somewhat close to mixed.

A \emph{log-Sobolev inequality} for the Markov chain is one way that such a phenomenon may be captured.  In particular, if the \emph{log-Sobolev constant} for the transposition chain on $\slice{\hamlevel}$ is $\varrho_\kappa$, it follows that
\begin{equation}    \label[ineq]{eqn:sse1}
    \Phi[A] \geq \half \varrho_\hamlevel \cdot \ln(1/\vol(A))
    \quad \text{for all nonempty subsets } A \subseteq \slice{\hamlevel}.
\end{equation}
So sets of constant volume must have conductance $\Omega(\varrho_\hamlevel)$, but sets of volume $2^{-\Theta(n)}$ (for example) must have conductance $\Omega(n \varrho_\hamlevel)$.  A known further consequence of a log-Sobolev inequality is a \emph{hypercontractive inequality}, which concerns expansion in the continuous-time version of the Markov chain.  It implies that if $\bsigma$ is the random permutation generated by performing the continuous-time chain for $t = \frac{\ln c}{2\varrho_\hamlevel}$ time --- i.e.,
\[
    \bsigma \text{ is the product of Poisson}\parens*{\frac{\ln c}{2\varrho_\hamlevel}}\text{ random transpositions, } c \geq 1
\]
--- then
\[
    \Pr_{\substack{\bu \sim A \\ \bsigma \sim \transpos{n}}}\bracks*{\bu^{\bsigma} \not \in A} \geq 1 - \vol(A)^{(c-1)/(c+1)} \quad \text{for all nonempty subsets } A \subseteq \slice{\hamlevel}.
\]
Thus again, if $\vol(A)$ is small, then the Markov chain will almost surely exit~$A$ after running for $\Theta(\varrho_\hamlevel^{-1})$ steps.

We remark that \Cref{eqn:sse1} is merely a \emph{consequence} of the log-Sobolev constant being~$\varrho_\hamlevel$.  It is not the case that $\varrho_\hamlevel$ is defined to be the largest constant for which \Cref{eqn:sse1} holds (for all~$A$) --- though this is a reasonable intuition.  Instead, $\varrho_\hamlevel$ is defined to be the largest constant for which a certain generalization of \Cref{eqn:sse1} to nonnegative functions holds; namely,
\begin{equation} \label[ineq]{eqn:logsob1}
    \E_{\substack{\bu \sim \pi \\ \btau \sim \transpos{n}}}\parens*{\sqrt{\phi(\bu)} - \sqrt{\phi(\bu^\btau)}}^2 \geq  \varrho_\hamlevel \cdot \KL{\phi \pi}{\pi}
     \quad \text{for all probability densities } \phi.
\end{equation}
(Here a \emph{probability density function} is a function $\phi\colon \slice{\hamlevel} \to \R^{\geq 0}$ satisfying $\E_\pi[\phi] = 1$, and $\KL{\phi \pi}{\pi}$ denotes the \emph{KL divergence} between distributions~$\phi \pi$ and~$\pi$.)  \Cref{eqn:logsob1} includes \Cref{eqn:sse1} by taking $\phi = 1_A/\vol(A)$. For more details, see \Cref{sec:prelims}.

Our main theorem in this work is a lower bound on the log-Sobolev constant for $\slice{\hamlevel}$:
\begin{theorem}                                     \label{thm:main}
    Let $\hamlevel \in \N_+^\colors$ satisfy $\hamlevel_1 +\cdots + \hamlevel_\colors = n$, and let $\varrho_\hamlevel$ denote the log-Sobolev constant for the transposition chain on the multislice $\slice{\hamlevel}$ (i.e., the largest constant for which \Cref{eqn:logsob1} holds).  Then
    \[
        \varrho_\hamlevel^{-1} \leq
        n \cdot \sum_{i=1}^\colors \tfrac12 \log_2(4n/\hamlevel_i).
    \]
\end{theorem}
The main case of interest for us is $n \longrightarrow \infty$ with $\colors = O(1)$ and $\hamlevel_i/n \geq \Omega(1)$ for each~$i$; in other words, when we are at a ``middling'' histogram of a high-dimensional multicube~$[\colors]^n$.  In this case our bound is $\varrho_\hamlevel \geq \Omega(1/n)$, which is the same bound that holds for the standard random walk on the Boolean cube. Thus for this parameter setting, the random transposition chain on the multislice enjoys all of the same small-set expansion properties as the Boolean cube (up to constants).

\paragraph{On the sharpness of \Cref{thm:main}.}  When $\colors$ is considered to be a constant, \Cref{thm:main} is sharp up to constant factors (which we did not attempt to optimize); i.e.,
\begin{equation}    \label{eqn:likely-truth}
    \varrho_{\hamlevel}^{-1} = \Theta(n) \cdot \log\parens*{\frac{n}{\min_i \{\hamlevel_i\}}} \quad \text{for } \colors = O(1).
\end{equation}
To see the upper bound on $\varrho_\hamlevel$, assume without loss of generality that $\colors = \argmin_i \{\hamlevel_i\}$, and take
\[
    A = \braces*{u \in \slice{\hamlevel} : u_j = \colors \text{ for all } j \in [\hamlevel_\colors]}.
\]
It is easy to compute that $\Phi[A] = \Theta(\hamlevel_\colors/n)$ and $\vol(A) = \binom{n}{\hamlevel_\colors}^{-1}$ (hence $\ln(1/\vol(A)) = \Theta(\hamlevel_\colors \log(n/\hamlevel_\colors)$)). Putting this into \Cref{eqn:sse1} shows the claimed upper bound on~$\varrho_\hamlevel$.

At the opposite extreme, when $\colors = n$ and $\hamlevel = (1, 1, \dots, 1)$, we have the random transposition walk on the symmetric group~$\symm{n}$.  In this case, \Cref{thm:main} as stated gives the poor bound of $\varrho_\hamlevel \geq \Omega(1/n^2 \log n)$, whereas the optimal bound is $\varrho_\hamlevel = \Theta(1/n \log n)$~\cite{DS96,LY98}.  In fact, our proof of \Cref{thm:main} (which generalizes that of~\cite{LY98}) can actually achieve the tight lower bound of $\varrho_\hamlevel \geq \Omega(1/n\log n)$ in this case.  However, we tailored our general bound for the case of~$\colors = O(1)$, and did not try to optimize for the most general scenario of~$\colors$ varying with~$n$.  A reasonable prediction might be that \Cref{eqn:likely-truth} always holds, up to universal constants, without the assumption of $\colors = O(1)$; we leave investigation of this for future work.

\subsection{Applications}
There are many known applications of log-Sobolev and hypercontractive inequalities in combinatorics and theoretical computer science (see, e.g.,~\cite[Ch.~9,~10]{OD14}).  In this paper we present four particular consequences of \Cref{thm:main} for analysis/combinatorics of Boolean functions on the multislice. We anticipate the possibility of several more.  Full details of these applications appear in \Cref{sec:applications}; here we describe them informally.

Throughout the remainder of this section, let us think of $n$ as large, of $\colors$ as constant, and let us fix a histogram~$\hamlevel$ (with $\hamlevel_1 + \cdots + \hamlevel_\colors = n$) satisfying $\hamlevel_i/n \geq \Omega(1)$ for all~$i$. For example, we might think of $\colors = 3$ and $\hamlevel = (n/3,n/3,n/3)$, so that $\slice{\hamlevel}$ consists of all ternary strings with an equal number of $1$'s, $2$'s, and $3$'s.  The \emph{isoperimetric problem} for $\slice{\hamlevel}$ would ask: for a given fixed~$0 < \alpha < 1$, which subset $A \subseteq \slice{\hamlevel}$ with $\vol(A) = \alpha$ has minimal ``edge boundary'', i.e., minimal~$\Phi[A]$?  (Here ``edge boundary'' is with respect to performing a single transposition, although in our Kruskal--Katona application we will relate this to the size of $A$'s ``shadows'' at neighboring multislices.)

We typically think of~$\alpha$ as ``constant'', bounded away from~$0$ and~$1$.  In our example with $\hamlevel = (n/3,n/3,n/3)$, when $\alpha = 1/3$ the isoperimetric minimizer is a ``dictator'' set like $A = \{u : u_1 = 1\}$; it has $\Phi[A] = \frac{4/3}{n-1}$.  The ``$99$\% regime'' version of the isoperimetric question would be: if $\Phi[A]$ is within a factor $1+o(1)$ of minimal, must $A$ be ``$o(1)$-close'' to a minimizer? This question will be considered in a companion paper. We will instead consider the ``$1$\% regime'' version of the isoperimetric question: if $\Phi[A]$ is at most $O(1)$ times the minimum, must~$A$ at least ``slightly resemble'' a minimizer?

To orient ourselves, first note that for constant~$\alpha$ (bounded away from $0$ and~$1$), the minimum possible value of $\Phi[A]$ among $A$ with $\vol(A) = \alpha$ is $\Theta(1/n)$; indeed, this follows from our \Cref{thm:main} and \Cref{eqn:sse1}.  From this fact, we will derive a multislice variant of the \textbf{Kruskal--Katona Theorem}.  Up to $O(1)$ factors, this minimum is achieved not just by ``dictator'' sets like $\{u \in \slice{(n/3,n/3,n/3)} : u_1 = 1\}$, but also by any ``junta'' set, meaning a set~$A$ for which absence or presence of $u \in A$ depends only on the colors $(u_j : j \in J)$ for a set $J \subseteq [n]$ of cardinality $c = O(1)$. It is not hard to see that if $A \subseteq \slice{\hamlevel}$ is such a \emph{$c$-junta}, then $\Phi[A] \leq O(c/n)$.  We may now ask: if $\Phi[A] \leq O(1/n)$, must~$A$ at least slightly ``resemble'' a junta?

We give two closely related positive answers to this question, as a consequence of our log-Sobolev inequality.  The first answer, a \textbf{KKL Theorem} for the multislice (cf.~\cite{KKL88}), follows  immediately from previous work~\cite{OW13a,OW13}.  It says that for any set with $\Phi[A] \leq O(1/n)$, there must exist some pair of coordinates $j,j' \in [n]$ with at least constant \emph{influence} on~$A$, where the influence of the transposition $(j\;j')$ on~$A$ is defined to be
\begin{equation}    \label{eqn:inf-set}
    \Inf_{(j\;j')}[A] = \Pr_{\bu \sim \pi}\bracks*{1_A\bigl(\bu\bigr) \neq 1_A\bigl(\bu^{(j\;j')}\bigr)}.
\end{equation}
It is the hallmark of a junta $A$ that every transposition $(j\;j')$ has either $\Inf_{(j\;j')}[A] = 0$ or $\Inf_{(j\;j')}[A] \geq \Omega(1)$.  In fact, mirroring the original KKL Theorem, our work shows that: (i)~if $\Phi[A] \leq c/n$ then there exists $(j\;j')$ with $\Inf_{(j\;j')}[A] \geq \exp(-O(c))$; (ii)~for \emph{any} $A \subseteq \slice{\hamlevel}$ with $\Omega(1) \leq \vol(A) \leq 1-\Omega(1)$, there exists $(j\;j')$ with $\Inf_{(j\;j')}[A] \geq \Omega\bigl(\frac{\log n}{n}\bigr)$.  From this, we can also derive a ``robust'' version of our Kruskal--Katona theorem (a~l\`a~\cite{OW13a}).

A closely related consequence of our work is a \textbf{Friedgut Junta Theorem} for the multislice (cf.~\cite{Fri98}), which follows (using a small amount of representation theory) from work of Wimmer~\cite{Wim14} (see also~\cite{Fil16a} for a different account).  It states that for any~$A$ with $\Phi[A] \leq c/n$, and any $\eps > 0$, there is a genuine $\exp(O(c/\eps))$-junta $A' \subseteq \slice{\hamlevel}$ that is $\eps$-close to~$A$, meaning \mbox{$\vol(A \symdiff A') \leq \eps$}. The junta theorem can also be generalized to real-valued functions, following the work of Bouyrie~\cite{Bou17}, with a worse dependence on $\eps$ in the exponent.

Finally, with a little more representation theory effort, we are able to derive from \Cref{thm:main} a \textbf{Nisan--Szegedy Theorem} for the multislice (cf.~\cite{NS94}), which is (roughly) an $\eps = 0$ version of the Friedgut Junta Theorem; this generalizes previous work on the Hamming slice~\cite{FI18b}.  It says that if $A \subseteq \slice{\hamlevel}$ is of ``degree~$k$'' --- meaning that its indicator function can be written as a linear combination of $k$-junta functions --- then $A$ must be an $\exp(O(k))$-junta itself.  (The $k = 1$ case of this theorem, with the conclusion that $A$ is a $1$-junta, was proven recently in~\cite{FI18a}.)

\subsection{Context and prior work}

In this section we review similar contexts where log-Sobolev inequalities and small-set expansion have been studied.

\paragraph{The Boolean cube.}  The simplest and best-known setting for these kinds of results is the Boolean cube~$\{0,1\}^n$ with the nearest-neighbour random walk.  The optimal hypercontractive inequality in this setting was proven by Bonami~\cite{Bon70}. Later, Gross~\cite{Gro75} introduced log-Sobolev inequalities, showed that they were equivalent to hypercontractive inequalities in this setting, and determined the exact log-Sobolev constant for the Boolean cube, namely $\varrho = 2/n$.  Gross also observed that all the same results also hold for Gaussian space in any dimension (recovering prior work of Nelson~\cite{Nel73}); Gaussian space is in fact a ``special case'' of the Boolean cube, by virtue of the Central Limit Theorem.  The Boolean cube also generalizes the well-studied \emph{Ehrenfest model} of diffusion~\cite{EE07}.

These inequalities for the Boolean cube, as well as the associated small-set expansion corollaries, have had innumerable applications in analysis, combinatorics, and theoretical computer science, in topics ranging from communication complexity to inapproximability; see, e.g.,~\cite{Led99} or~\cite[Chapters 9--11]{OD14}.

A different line of work sought to determine the exact minimum value of $\Phi[A]$ in terms of the size of $A$. This challenge, known as the \emph{edge isoperimetric problem}, has been solved by Harper~\cite{Harper64}, Lindsey~\cite{Lindsey64}, Bernstein~\cite{Bernstein67}, and Hart~\cite{Hart76}, who have shown that the optimal sets are initial segments of a lexicographic ordering of the vertices of the Boolean cube. Recently Ellis, Keller and Lifshitz gave a new proof of the edge isoperimetric inequality using the Kruskal--Katona Theorem~\cite{EKL17}. The same set of authors also recently proved a stability version of the edge isoperimetric inequality in the 99\% regime~\cite{EKL18}.

Returning to log-Sobolev inequalities, an extraordinarily helpful feature of the random walk on the Boolean cube is that it is a \emph{product Markov chain}, with a stationary distribution that is \emph{independent} across the $n$~coordinates.  Because of this, a simple induction lets one immediately reduce the log-Sobolev (and hypercontractivity) analysis to the base case of~$n = 1$.

\paragraph{Other product chains.}  For \emph{any} product Markov chain, one can similarly reduce the analysis to the~$n = 1$ case. In general, let $\nu$ be a probability distribution of full support on~$[\colors]$, and consider the Markov chain on $[\colors]^n$ in which a step from $u \in [\colors]^n$ consists of choosing a random coordinate $\bj \sim [n]$ and replacing $u_\bj$ with a random draw from~$\nu$.  The invariant distribution for this chain is the product distribution~$\nu^{\otimes n}$.   Though the $n = 1$ case of this chain is, in a sense, trivial --- it mixes perfectly in one step --- it is not especially easy to work out the optimal log-Sobolev constant.  Nevertheless, Diaconis and Saloffe-Coste~\cite{DS96} showed that for the $n = 1$ chain, the log-Sobolev constant is
\begin{equation}    \label{eqn:base-log-sob}
    \varrho^{\mathrm{triv}}_\nu = 2\frac{q-p}{\ln q - \ln p}, \quad \text{where } p = \min_{i \in [\colors]} \{\nu(i)\},\ q = 1-p.
\end{equation}
It follows immediately that the log-Sobolev constant in the general-$n$ case is $\varrho^{\mathrm{triv}}_\nu/n$.  In particular, if $\hamlevel_1 + \cdots + \hamlevel_\colors = n$ and $\nu(i) = \hamlevel_i/n$, then $\nu^{\otimes n}$ resembles the uniform distribution $\pi_\hamlevel$ on $\slice{\hamlevel}$, and the product chain on $[\colors]^n$ somewhat resembles the random transposition chain on $\slice{\hamlevel}$.  This gives credence to the possibility that \Cref{eqn:likely-truth} may hold with absolute constants for any~$\colors$.

\paragraph{The Boolean slice / Bernoulli--Laplace model / Johnson graph.}  Significant difficulties arise when one moves away from product Markov chains.  One of the simplest steps forward is to the Boolean slice.  This is the $\colors = 2$ case of the Markov chains studied in this paper, with the ``balanced'' case of $\hamlevel = (n/2,n/2)$ being the most traditionally studied.  This Markov chain is also equivalent to the Bernoulli--Laplace model for diffusion between two incompressible liquids, and to the standard random walk on \emph{Johnson graphs}; taking multiple steps in the chain is similar to the random walk in \emph{generalized Johnson graphs}. The chain has been studied in wide-ranging contexts, from genetics~\cite{Mor58}, to child psychology~\cite{PI76}, to computational learning theory~\cite{OW13a}.  An asymptotically exact analysis of the time to stationarity of this Markov chain was given by Diaconis and Shahshahani~\cite{DS87}, using representation theory. However, the log-Sobolev constant for the chain took a rather long time to be determined; it was left open in Diaconis and Saloff-Coste's 1996 survey~\cite{DS96} before finally being determined (up to constants) by Lee and Yau in 1998~\cite{LY98}.  This sharp log-Sobolev inequality, and its attendant hypercontractivity and small-set expansion inequalities, have subsequently been used in numerous applications --- for the Kruskal--Katona and Erd\H{o}s--Ko--Rado theorems in combinatorics~\cite{OW13a,DK16,FKMW18}, for computational learning theory~\cite{Wim09,OW13a}, for property testing~\cite{Mos14}, and for generalizing classic ``analysis of Boolean functions'' results~\cite{OW13a,OW13,Fil16a,Fil16b,FM16,FKMW18,Bou18}.

\paragraph{The Grassmann graph.}  One direction of generalization for the Johnson graphs are their ``$q$-analogues'', the \emph{Grassmann graphs}; understanding this Markov chain was posed as an open problem even in the early work of Diaconis and Shahshahani~\cite[Example~2]{DS87}.  For a finite field~$\F$ and integer parameters $n \geq k \geq 1$, the associated Grassmann graph has as its vertices all $k$-dimensional subspaces of~$\F^n$, with two subspaces connected by an edge if their intersection has dimension~$k-1$.  Understanding small-set expansion (and lack thereof) in the Grassmann graphs was central to the very recent line of work that positively resolved the $2$-to-$2$ Conjecture~\cite{KMS17,DKKMS18a,DKKMS18b,BKS18,KMS18} (with the analogous problems on the Johnson graphs serving as an important warmup~\cite{KMMS18}).  Still, it seems fair to say that the mixing properties of the Grassmann graph are far from being fully understood.

\paragraph{The multislice.}  We now come to the multislice, the other natural direction of generalization for the Johnson graphs, and the subject of the present paper.  One can see the multislice as a generalization of the Bernoulli--Laplace model, modeling diffusion between three or more liquids.  As well, the space of functions $f\colon \slice{\hamlevel} \to \R$, together with the action of $\symm{n}$ on $\slice{\hamlevel}$, is precisely the \emph{Young permutation module} $M^\hamlevel$ arising in the representation theory of the symmetric group. Understanding the mixing properties of the $\slice{\hamlevel}$ Markov chain with random transpositions was suggested as an open problem several times~\cite{DS87}, \cite[p.~59]{Dia88}, \cite{FI18b}.  The multislice has also played a key combinatorial role in problems in combinatorics, such as the Density Hales--Jewett problem (where~$\colors = 3$ was the main case under consideration)~\cite{Pol12}.

Although it might at first appear to be a simple generalization of the Boolean slice, there are several fundamental impediments that arise when moving from $\colors = 2$ even to $\colors = 3$.  These include: the fact that a Hamming slice disconnects the nearest-neighbour graph in $[2]^\colors$ but not in~$[3]^\colors$; the fact that one can introduce just \emph{one} variable per coordinate when representing functions \mbox{$[2]^\colors \to \R$} as multilinear polynomials; the fact that $2$-row irreps of $\symm{n}$ (Young diagrams) are completely defined by the number of boxes not in the first row; and, the fact that  when $\colors \geq 3$, the decomposition of the permutation module~$M^\hamlevel$ into irreps has multiplicities.  The last of these was the main difficulty to be overcome in Scarabotti's work~\cite{Sca97} giving the asymptotic mixing time for the transposition walk on balanced multislices~$\slice{(n/\colors, \dots, n/\colors)}$ (see also~\cite{DH02,ST10}). It also prevents the multislice from forming an association scheme.

For the purposes of this paper, the main difficulty that arises when analyzing the log-Sobolev inequality is the following: when $\colors = 2$, any nontrivial step in the Markov chain (switching a~$1$ and a~$2$) has the property that the histogram within $[\colors]^{n-2}$ of the unswitched colors is always the same: $(\hamlevel_1-1, \hamlevel_2 -1)$.  By contrast, once $\colors \geq 3$, the multiple ``kinds'' of transpositions (switching a~$1$ and a~$2$, or a~$1$ and a~$3$, or a~$2$ and a~$3$, etc.)\ lead to differing histograms within $[\colors]^{n-2}$ for the unswitched colors.  This significantly complicates inductive arguments.

\paragraph{The symmetric group and beyond.}  Finally, we mention that analysis of the multislice can also be motivated simply as a necessary first step in a full understanding of spectral analysis on the symmetric group and other algebraic structures, an opinion also espoused in, e.g.,~\cite{CFR11}. Such structures include classical association schemes such as polar spaces and bilinear forms, matrix groups such as the general linear group, and the $q$-analog of the multislice.

\section{Preliminaries} \label{sec:prelims}

\subsection{Definitions relevant for the log-Sobolev inequality}
Given parameters $\colors \in \N_+$ (number of colors) and $n \in \N_+$ (number of coordinates), our objects of study in this paper are \emph{multislices}, parametrized by a histogram $\hamlevel \in \N_+^\colors$ satisfying $\hamlevel_1 + \cdots + \hamlevel_\colors = n$:
\[
    \slice{\hamlevel} = \braces*{u \in [\colors]^n : \hamwt{i}{u} = \hamlevel_i \text{ for all } i \in [\colors]}.
\]
We will only consider multislices with at least two colors; in other words, $\colors \geq 2$.

We introduce the inner product space of functions on~$\slice{\hamlevel}$,
\[
    M^\hamlevel = \braces*{f\colon \slice{\hamlevel} \to \R}, \quad \text{with } \la f, g \ra = \E_{\bu \sim \pi}[f(\bu)g(\bu)],
\]
where
$\pi = \pi_\hamlevel$ denotes the uniform distribution on the multislice $\slice{\hamlevel}$.

Let $\Ker$ denote the \emph{transition / Markov operator} on $M^\hamlevel$ associated to the transposition random walk, defined by
\[
    \Ker f(u) = \E_{\btau \sim \transpos{n}}[f(u^\btau)],
\]
where $\transpos{n}$ consists of all $\binom{n}{2}$ transpositions.
Let $\Lap$ denote the \emph{Laplacian operator} $\bbone - \Ker$ (where $\bbone$ is the identity operator).  Then the \emph{energy} (or \emph{Dirichlet form}) of $f\colon \slice{\hamlevel} \to \R$ is
\[
    \Energy[f] = \la f, \Lap f \ra = \half \E_{\bu \sim \bv}\bracks*{\parens*{f(\bu) - f(\bv)}  ^2},
\]
where we have introduced the notation $\bu \sim \bv$ to denote that $(\bu, \bv)$ is a random edge in the Schreier graph; equivalently, $\bu \sim \pi$ and $\bv = \bu^\btau$ for $\btau \sim \transpos{n}$.  One may check that if $A \subseteq \slice{\hamlevel}$, then
\[
    \Energy[1_A] = \vol(A) \cdot \Phi[A],
\]
where $1_A \in M^\hamlevel$ denotes the $0$/$1$-indicator of~$A$.
(Recall that $\Phi[A]$ is the probability, over $\bu \sim A$ and $\btau \sim \transpos{n}$, that $\bu^\btau \notin A$.)

Before formally defining the log-Sobolev inequality for the transposition chain on $\slice{\hamlevel}$, we recall first its simpler counterpart, the \emph{Poincar\'e inequality}. For $f \in M^\hamlevel$, this is
\[
    \Energy[f] \geq \lambda_1 \cdot \Var_\pi[f],
\]
where $\lambda_1$ is the \emph{spectral gap}; i.e., the lowest eigenvalue of~$\Lap$ other than the trivial $\lambda_0 = 0$. For the transposition chain on $\slice{\hamlevel}$ it is known that $\lambda_1 = \frac{2}{n-1}$, with ``dictator'' functions and other ``degree-$1$'' functions providing the tight examples; see \Cref{cor:spectral-gap} in \Cref{sec:constant-degree}.

As for the log-Sobolev inequality, it is typically defined as
\[
    \Energy[f] \geq \tfrac12 \varrho \cdot \Ent[f^2];
\]
the largest constant~$\varrho$ that is acceptable for all $f \in M^\hamlevel$ being termed the log-Sobolev constant for~$\slice{\hamlevel}$.  Here $\Ent[g] = \E_{\pi}[g \ln g] - \E_{\pi}[g]\cdot \ln \E_{\pi}[g]$.  We remark that Diaconis and Saloff-Coste~\cite{DS96} showed that $\varrho \leq \lambda_1$ always holds.

In this work we will prefer a slightly different (equivalent) definition for the log-Sobolev inequality, used in~\cite{LY98}.  It's easy to see that replacing $f$ with $\abs{f}$ does not change $\Ent[f^2]$ but can only decrease $\Energy[f]$.  Thus in the log-Sobolev inequality it suffices to consider nonnegative~$f$.  Also, the inequality is $2$-homogeneous (both sides are multiplied by~$c^2$ when $f$ is multiplied by~$c$); thus it suffices to consider nonnegative~$f$ with $\E_\pi[f^2] = 1$. We write a general such~$f$ as $\sqrt{\phi}$, where $\phi\colon \slice{\hamlevel} \to \R^{\geq 0}$ satisfies $\E_\pi[\phi] = 1$.  We call such a $\phi$ a \emph{probability density function}, thinking of it as a relative probability density with respect to the uniform distribution $\pi = \pi_{\hamlevel}$.  In other words, we associate $\phi$ to the probability distribution in which $u \in \slice{\hamlevel}$ has probability mass $\phi(u)\pi(u)$.  Now when $\phi = f^2$, we have
\[
    \Ent[f^2] = \Ent[\phi] = \E_{\pi}[\phi \ln \phi] - \E_{\pi}[\phi]\cdot \ln \E_{\pi}[\phi] = \E_{\pi}[\phi \ln \phi],
\]
where we used $\ln \E_{\pi}[\phi] = \ln 1 = 0$.  This quantity is precisely the \emph{Kullback--Leibler divergence} (or \emph{relative entropy}) between the distribution $\phi \pi$ and the distribution~$\pi$, denoted $\KL{\phi \pi}{\pi}$.  Thus we have shown that the usual formulation of the log-Sobolev inequality is equivalent to
\begin{equation}    \label[ineq]{eqn:equiv-ls}
    \Energy\bracks*{\sqrt{\phi}} \geq \tfrac12 \varrho_\hamlevel \cdot \KL{\phi \pi}{\phi}
	\end{equation}
for all probability densities $\phi$ on~$\slice{\hamlevel}$, as stated in \Cref{eqn:logsob1}.

\subsection{Hypercontractivity, influences, and other preliminaries for our applications}
In this section we make some further definitions, which will be useful for our applications of the log-Sobolev inequality.
We first decompose $\Lap$ into its components on each transposition $\tau \in \transpos{n}$, introducing the operator $\Lap_\tau$ defined by
\[
    \Lap_{\tau}f(x) = f(u)-f(u^{\tau})
\]
Note that
\[
\Lap = \avg_{\tau \in \transpos{n}} \Lap_{\tau}.
\]
Now for $f : \slice\hamlevel \to \R$ we define  the \emph{influence} of  transposition $\tau \in \transpos{n}$ on~$f$ to be
\[
    \Inf_{\tau}[f] = \langle f, \Lap_\tau f \rangle = \langle  \Lap_\tau f, f \rangle = \tfrac12 \norm{\Lap_{\tau}f}_2^2
\]
(where we are using norm notation $\norm{f}_p = \E[|f|^p]^{1/p}$).  Note that if $A \subseteq \slice\hamlevel$ then we have the following combinatorial interpretation, agreeing with our notation from \Cref{eqn:inf-set}:
\[
    \Inf_\tau[A] = \Inf_\tau[1_A] = \Pr_{\bu \sim \pi_\hamlevel} [\bu \in A,\ \bu^\tau \not \in A] = \Pr_{\bu \sim \pi_\hamlevel} [\bu \not \in A,\ \bu^\tau \in A].
\]
For general $f : \slice\hamlevel \to \R$ we introduce the following additional notation, for the \emph{average influence} (equivalent to \emph{energy}), \emph{total influence}, and \emph{maximum influence} of~$f$:
\[
    \calE[f] = \avg_{\tau\in\transpos{n}} \Inf_{\tau}[f], \qquad
    \Inf[f] = \sum_{\tau \in \transpos{n}} \Inf_\tau[f]  = \tbinom{n}{2} \calE[f] \qquad
    \calM[f] = \max_{\tau\in\transpos{n}} \Inf_{\tau}[f].
\]

As recounted in the important survey of Diaconis and Saloff-Coste~\cite{DS96}, there is an equivalence between log-Sobolev inequalities and hypercontractivity for reversible Markov chains.
To explain what hypercontractivity means in this abstract setting, we first define the continuous-time analog of the random transposition walk. This is the continuous-time Markov chain, running from time $t=0$ to $t=\infty$, in which (informally) in any interval of infinitesimal length~$dt$, one performs step of the random transposition chain with probability~$dt$.

More formally, we can provide the following alternative description: if we initialize the continuous-time Markov chain at state~$u$, then its state $\bu_t$ at time $t \geq 0$ is obtained by performing $\mathrm{Poisson}(t)$ random transpositions on $u$. From this definition we may define  \emph{noise operator} (or \emph{heat kernel})~$\Heat_t$ on~$M^\hamlevel$:
\[
    \Heat_t f(u) = \E[f(\bu_t)].
\]
It's well known that we can also express $\Heat_t$ in terms of the Laplacian operator, $\Heat_t = e^{-t\Lap}$.

Diaconis and Saloff-Coste~\cite[Theorem 3.5(ii)]{DS96} show that the log-Sobolev inequality implies hypercontractivity:
\begin{theorem} \label{thm:hypercontractivity}
For $q \geq 2$ and $\displaystyle t = \frac{\ln (q-1)}{2\varrho_\hamlevel}$, the following hold for all $f \in M^\hamlevel$:
\[
    \|\Heat_t f\|_q \leq \|f\|_2,  \qquad
    \|\Heat_t f\|_2 \leq \|f\|_{q'},
\]
where $q'$ is the H\"older conjugate of~$q$ (meaning $1/q + 1/q' = 1$).
\end{theorem}
\noindent (The first inequality directly appears in~\cite{DS96}; the second statement appears only implicitly.  It is a consequence of $\Heat_t f$ being a self-adjoint operator; see, e.g.,~\cite[Prop.~9.19]{OD14}.)

In all our applications, we actually use hypercontractivity, rather than log-Sobolev directly.

\subsection{Two numerical lemmas}
Here we give some two elementary numerical lemmas we'll need for our proofs.  We start by computing and bounding the inverse moment of a hypergeometric random variable.  The following simple fact may well be in the literature; the most relevant citation we found was~\cite{Gov64}:
\begin{lemma}                                       \label{lem:inv-hypergeom}
    Let $\bX \sim \mathrm{Hypergeometric}(N,K,n)$.\footnote{In other words, $\bX$ is the number of white balls among $n$ random balls drawn (without replacement) from an urn containing $N$ balls in total, $K$ of which are white.}  Then, writing $p = 1- \binom{N-K}{n+1}/\binom{N+1}{n+1}$, it holds that
    \[
        \E\bracks*{\frac{1}{\bX+1}} = p \frac{N+1}{(n+1)(K+1)} \leq \frac{N+1}{(n+1)(K+1)}.
    \]
\end{lemma}
\begin{proof}
    Take an urn with $N+1$ balls, $K+1$ of them white, one of the white balls being ``special''.  Consider an experiment in which we draw $n+1$ balls without replacement, and then choose a random white ball among the ones drawn (if any).  A ``success'' occurs if the randomly chosen white ball is the special one.

    A necessary condition for the experiment to succeed is that the special white ball was chosen at all, which happens with probability $\frac{n+1}{N+1}$.  Assuming that this happened, the number of remaining white balls drawn is distributed as $\bX \sim \mathrm{Hypergeometric}(N,K,n)$.  Thus the probability that we finally choose the special ball is $\frac{n+1}{N+1} \E[\frac{1}{\bX+1}]$.

    Let us now think of the experiment in a different way.  A necessary condition for success is that at least one white ball is drawn, which happens with probability~$p$.  Given that this occurs, the finally chosen white ball is just a random white ball (among all $K+1$ white balls), so the probability that it is the special one is exactly $\frac{1}{K+1}$.  Thus
    \[
        \frac{n+1}{N+1} \E\left[\frac{1}{X+1}\right] = \frac{p}{K+1},
    \]
    completing the proof of the lemma.
\end{proof}

Next, we give a bound on the log-Sobolev constant from \Cref{eqn:base-log-sob} that is more tractable.
\begin{lemma}   \label{lem:triv-bound}
    Let $\nu$ be a probability distribution of full support on~$[\colors]$.  Then the log-Sobolev constant~$\varrho^{\mathrm{triv}}_\nu$ for the associated trivial Markov chain, given in \Cref{eqn:base-log-sob}, satisfies the bound
    \[
        (\varrho^{\mathrm{triv}}_\nu)^{-1} \leq \frac12 \sum_{i=1}^\colors \lg\parens*{\frac{1}{\nu_i}},
    \]
    where $\lg$ denotes $\log_2$.
\end{lemma}
\begin{proof}
    Let $p = \min_{i \in [\colors]} \{\nu(i)\}$, and assume without loss of generality that this minimum is achieved by $i = \colors$.
    Recall the formula of Diaconis and Saloffe-Coste~\cite{DS96},
\[
    \varrho^{\mathrm{triv}}_\nu = 2\frac{q-p}{\ln q - \ln p}, \quad \text{where } q = 1-p.
\]
	What we need to show is
    \[
        \frac{\ln(1/p) - \ln(1/(1-p))}{1-2p} \leq \lg(1/p) + \sum_{i=1}^{\colors-1} \lg(1/\nu(i)).
    \]
    By convexity of $t \mapsto \lg(1/t)$ for $t \in (0,1]$, the right-hand side above is at least
    \[
        \lg(1/p) + (\colors-1)\lg\parens*{\frac{\colors-1}{1-p}} \geq \lg(1/p) + \lg(1/(1-p)),
    \]
    where the second inequality used $\colors \geq 2$.  Thus it suffices to show
    \[
        \ln(1/p) - \ln(1/(1-p)) \leq (1-2p)(\lg(1/p) + \lg(1/(1-p))), \quad 0 < p \leq 1/2.
    \]
    (This inequality is simply the lemma we are trying to prove, restricted to the case $\colors = 2$.)

    Write $p = 1/2 - \delta/2$, where $\delta \in [0,1)$.  Both sides of the above inequality are zero for $\delta = 0$; thus to establish the inequality it suffices to show the right-hand side's derivative is at least the left-hand side's.  Taking derivatives, we need to show
    \[
        \frac{2}{1-\delta^2} \leq \lg\parens*{\frac{4}{1-\delta^2}} + \frac{2\delta^2}{(\ln 2)(1-\delta^2)}, \quad 0 \leq \delta < 1.
    \]
    Multiplying this by $\frac{1-\delta^2}{2} > 0$ gives
    \[
        1 \leq 1-\delta^2 + \frac12(1-\delta^2)\lg\parens*{\frac{1}{1-\delta^2}} + \frac{\delta^2}{\ln 2} \quad\iff\quad 0 \leq \parens*{\frac{1}{\ln 2}-1}\delta^2 + \frac12(1-\delta^2)\lg\parens*{\frac{1}{1-\delta^2}},
    \]
    which is evidently true as $\frac{1}{\ln 2}-1 > 0$ and $0 < 1-\delta^2 \leq 1$.
\end{proof}

\section{Bounding the log-Sobolev constant}
In this section we will frequently identify a histogram $\hamlevel \in \N_+^\colors$ with the associated \emph{multiset} of colors, namely the multiset with $\hamlevel_i$ copies of~$i$ for each $i \in [\colors]$.  Thus we may write $n = |\hamlevel| = \hamlevel_1 + \cdots + \hamlevel_\colors$.  We will also use the notation $\bi \sim \hamlevel$ to mean that $\bi$ is chosen uniformly at random from the multiset~$\hamlevel$; i.e., according to the probability distribution on $[\colors]$ in which~$i$ has probability~$\hamlevel_i/n$.  We will write $\ul{\hamlevel}$ for this probability distribution, and will need to refer to the log-Sobolev constant $\varrho^{\mathrm{triv}}_{\ul{\hamlevel}}$ from \Cref{eqn:base-log-sob}.

Let us introduce one more piece of notation:
if $\varrho_\hamlevel$ denotes the optimal log-Sobolev constant for the transposition Markov chain on the multislice $\slice{\hamlevel}$, we will write
\[
    \invLS{\hamlevel} = \varrho_{\hamlevel}^{-1}.
\]
Thus the goal of our \Cref{thm:main}
is to upper-bound $\invLS{\hamlevel}$.  The midpoint of the proof will be establishing the following inductive bound:
\begin{lemma}                                     \label{lem:induction-structure}
    Let $\hamlevel \in \N_+^\colors$ with $n = |\hamlevel| > 2$. Then
    \begin{equation}    \label[ineq]{eqn:recursion}
        \invLS{\hamlevel} \leq
            \tfrac{n-1}{n}\cdot (\varrho^{\mathrm{triv}}_{\ul{\hamlevel}})^{-1} + \max_{\substack{i_1,i_2 \in [\colors] \\ \textnormal{distinct}}} \braces*{\E_{\bi \sim \hamlevel \setminus \{i_1,i_2\}} \invLS{\hamlevel \setminus \bi}}.
    \end{equation}
\end{lemma}
Given \Cref{lem:induction-structure}, the deduction of  \Cref{thm:main}  will be elementary, though not completely straightforward; this is in \Cref{sec:finish-recursion}.  As for the deduction of \Cref{lem:induction-structure} itself, it will mostly follow the proof Lee and Yau used~\cite{LY98} to analyze $\varrho_\hamlevel$ in the cases $\colors = 2$ (the Hamming slice) and $\colors = n$ (the symmetric group~$\symm{n}$).  Notice, however, that in both of these cases the ``max'' appearing in \Cref{eqn:recursion} becomes superfluous.  In  the $\colors = 2$ case, the only possibility for $\{i_1, i_2\}$ is $\{1,2\}$, so we have the much simpler recursion
\[
    \invLS{\hamlevel} \leq
        \tfrac{n-1}{n}\cdot (\varrho^{\mathrm{triv}}_{\ul{\hamlevel}})^{-1} + \E_{\bi \sim \hamlevel \setminus \{1,2\}} \invLS{\hamlevel \setminus \bi}.
\]
In the $\colors = n$ case (meaning $\hamlevel_i = 1$ for all $i \in [\colors]$), we see by symmetry that every choice of $i_1, i_2$ leads to an isomorphic subproblem, that of bounding the log-Sobolev constant for~$\symm{n-1}$.  That is, we have the even simpler recursion
\[
    \invLS{(1^n)} \leq
        \tfrac{n-1}{n}\cdot (\varrho^{\mathrm{triv}}_{(1/n, \dots, 1/n)})^{-1} + \invLS{(1^{n-1})} \lesssim \tfrac12 \ln n + \invLS{(1^{n-1})},
\]
where the asymptotic inequality used \Cref{eqn:base-log-sob}.  This recursion straightforwardly yields the known bound for the symmetric group, $\invLS{(1^N)} \lesssim \frac{1}{2} n \ln n$.  A key point of our work is recognizing that one can use the Lee--Yau methodology to obtain \Cref{eqn:recursion}, and that despite its somewhat complicated form, this recursion can be solved to yield a good bound.

\subsection{Proving \texorpdfstring{\Cref{lem:induction-structure}}{Lemma \ref{lem:induction-structure}}}
Although much of the proof of \Cref{lem:induction-structure} is from~\cite{LY98}, we recapitulate it here for completeness and clarity.  Fix $\hamlevel \in \N_+^\colors$ with $n = |\hamlevel| > 2$.  Recall from \Cref{eqn:equiv-ls} that $\invLS{\hamlevel}$ is the smallest constant such that
\begin{equation}    \label[ineq]{eqn:begin}
    \KL{\phi \pi}{\pi} \leq 2 \invLS{\hamlevel} \cdot \Energy\bracks*{\sqrt{\phi}}
\end{equation}
holds for all probability densities on~$\slice{\hamlevel}$.

The way we recursively bound $\invLS{\hamlevel}$ involves applying the chain rule for KL divergence to~\mbox{$\KL{\phi \pi}{\pi}$}. We will set up the notation for invoking the chain rule with respect to the $n$th coordinate.  In fact, we will eventually apply it for \emph{each} coordinate $k \in [n]$, and then take expectations over a uniformly random~$k$.  However, it will be notationally convenient to focus just on the $k = n$ case.

To this end, given a probability distribution~$\xi$ on $\slice{\hamlevel}$ (which will be either $\phi \pi$ or $\pi$), we will write~$\xi_n$ to denote its marginal on the $n$th coordinate (a probability distribution on~$[\colors]$). Also, given a particular $a \in [\colors]$, we will write $\xi_{|a}$ to denote $\xi$'s distribution on $\slice{\hamlevel}$ conditioned on the last coordinate having color~$a$.

To begin the analysis of \Cref{eqn:begin}, let us write
\[
    \phantom{\quad \text{(a probability distribution on $\slice{\hamlevel}$)}} \psi = \phi \pi \quad \text{(a probability distribution on $\slice{\hamlevel}$)}
\]
and then apply the chain rule for KL divergence with respect to the $n$th coordinate:
\[
    \KL{\psi}{\pi} = \underbrace{\KL{\psi_n}{\pi_n}\vphantom{\E_{\ba \sim \psi_n}\bracks*{\KL{\psi_{|\ba}}{\pi_{|\ba}}}}}_{\text{MARGINAL}_n} + \underbrace{\E_{\ba \sim \psi_n}\bracks*{\KL{\psi_{|\ba}}{\pi_{|\ba}}}}_{\text{CONDITIONAL}_n}.
\]
We will now bound~MARGINAL${}_n$ and~CONDITIONAL${}_n$.  In each case we will prove a bound that has a certain ``dependence on the $n$th coordinate''.  We will then  remark that we could have equally well proved an analogous bound involving the $k$th coordinate, for any $k \in [n]$.  Finally, we will average this analogous bound over all $k \in [n]$.  Adding the two averaged bounds from the MARGINAL and CONDITIONAL cases yields a valid upper bound on $\KL{\psi}{\pi}$,
\begin{equation}    \label[ineq]{eqn:fika0}
    \KL{\psi}{\pi} \leq \avg_{\bk \sim [n]}\Bigl\{\text{bound on MARGINAL}_{\bk}\Bigr\} + \avg_{\bk \sim [n]}\Bigl\{\text{bound on CONDITIONAL}_{\bk}\Bigr\}.
\end{equation}
From this we will derive a recursive upper bound on $\invLS{\hamlevel}$ via \Cref{eqn:begin}.
\subsubsection{Bounding \texorpdfstring{MARGINAL${}_n$}{MARGINALn}}
Our bound on MARGINAL${}_n$ is from~\cite{LY98}.  We think of the trivial Markov chain on $[\colors]$ with invariant distribution~$\pi_n$, noting that $\pi_n$ is nothing more than $\ul{\hamlevel}$.  Applying the associated log-Sobolev inequality, we get
\begin{equation}    \label[ineq]{eqn:get-started}
    \KL{\psi_n}{\pi_n} \leq 2\cdot (\varrho^{\mathrm{triv}}_{\ul{\hamlevel}})^{-1} \cdot \Energy_{\pi_n}\bracks*{\sqrt{\psi_n/\pi_n}},
\end{equation}
where we wrote $\Energy_{\pi_n}$ to denote the energy functional for the trivial Markov chain. It is simple to check that the function $\psi_n/\pi_n$ on~$[\colors]$ is just $a \mapsto \E_{\bu \sim \pi_{|a}}[\phi(\bu)]$.  Thus
\begin{equation}    \label{eqn:put-me-in}
    \Energy_{\pi_n}\bracks*{\sqrt{\psi_n/\pi_n}} = \frac12 \E_{\ba, \bb \sim \pi_n}\parens*{\sqrt{ \E_{\bu \sim \pi_{|\ba}}[\phi(\bu)]} - \sqrt{ \E_{\bv \sim \pi_{|\bb}}[\phi(\bv)]}}^2.
\end{equation}
We would prefer if the two inner expectations here were over the same probability space.  To that end, observe that for any $a,b \in [\colors]$ (not necessarily distinct), we can make a draw $\bv \sim \pi_{|b}$ in the following unusual way.  First, draw $\bu \sim \pi_{|a}$.  Next, draw $\bj \sim \bu^{-1}(b)$, where we have introduced the notation $u^{-1}(b)$ for the set of all coordinates~$j$ with $u_j = b$.  Finally, form $\bv = \bu^{(\bj\;n)}$.  (Here we are abusing  notation by allowing the possibility of $\bj = n$, so that the ``transposition'' $(\bj\;n)$ may be the identity.)  Thus we have
\[
    \E_{\bv \sim \pi_{|b}}[\phi(\bv)] = \E_{\substack{\bu \sim \pi_{|a} \\ \bj \sim \bu^{-1}(b)}}[\phi^{(\bj\;n)}(\bu)],
\]
where use the notation $\phi^{\tau}(u) \coloneqq \phi(u^{\tau})$.  Putting this into \Cref{eqn:put-me-in} (and also ``pointlessly'' choosing $\bj \sim \bu^{-1}(\bb)$ in the first inner expectation) we get
\begin{equation}    \label{eqn:weird2}
    \Energy_{\pi_n}\bracks*{\sqrt{\psi_n/\pi_n}} = \frac12 \E_{\ba, \bb \sim \pi_n}
    \abs*{\sqrt{ \E_{\substack{\bu \sim \pi_{|\ba} \\ \bj \sim \bu^{-1}(\bb)}}[\phi(\bu)]} - \sqrt{ \E_{\substack{\bu \sim \pi_{|\ba} \\ \bj \sim \bu^{-1}(\bb)}}[\phi^{(\bj\;n)}(\bu)]}}^2.
\end{equation}
Considering the quantity inside the outer expectation, we further use
\[
    \abs*{\sqrt{\E_{\substack{\bu \sim \pi_{|\ba} \\ \bj \sim \bu^{-1}(\bb)}}\bracks*{\phi(\bu)}} - \sqrt{\E_{\substack{\bu \sim \pi_{|\ba} \\ \bj \sim \bu^{-1}(\bb)}}\bracks*{\phi^{(\bj\;n)}(\bu)}}}
    = \abs*{ \norm*{\sqrt{\phi}}_2 - \norm*{\sqrt{\phi^{(\noarg\; n)}}}_2}
    \leq  \norm*{\sqrt{\phi} - \sqrt{\phi^{(\noarg\;n)}}}_2,
\]
where $\norm{\noarg}_2$ is the $2$-norm defined by the distribution on $(\bu, \bj)$, and we used the triangle inequality.  Putting this back into \Cref{eqn:weird2} yields
\begin{equation}    \label[ineq]{eqn:weird3}
    \Energy\bracks*{\sqrt{\mu_n/\pi_n}}
    \leq \frac12 \E_{\ba, \bb \sim \pi_n}\bracks*{\norm*{\sqrt{\phi} - \sqrt{\phi^{(\noarg\;n)}}}_2^2}
    = \E_{\ba, \bb \sim \pi_n}  \E_{\substack{\bu \sim \pi_{|\ba} \\ \bj \sim \bu^{-1}(\bb)}}
            \bracks*{e\parens*{\sqrt{\phi};\bu, \bu^{(\bj\;n)}}},
\end{equation}
where we have introduced the shorthand
\[
    e(f;u,v) = \tfrac12(f(u) - f(v))^2.
\]
In the right-hand expectation in \Cref{eqn:weird3}, $\bu$ is simply distributed according to~$\pi$.  Furthermore, suppose we fix any outcome $\bu = u$.  Let us consider the joint distribution of $\bb \sim \pi_n$ and $\bj \sim u^{-1}(\bb)$.  Since $u \in \slice{\hamlevel}$, we could equivalently form $\bb$ by choosing a random color within~$\bu$.  But in this case, $\bj$~is formed by first choosing a random color within~$u$, and then taking a random coordinate where $u$ has that same color. It is clear that the resulting distribution on~$\bj$ is simply uniformly random on~$[n]$.  Thus the right-hand side of \Cref{eqn:weird3} is simply
\[
    \E_{\bu \sim \pi} \E_{\bj \sim [n]}
            \bracks*{e\parens*{\sqrt{\phi};\bu, \bu^{(\bj\;n)}}}.
\]
Putting this into \Cref{eqn:weird3} and then into \Cref{eqn:get-started}, we conclude
\[
    \text{MARGINAL}_n \leq 2\cdot (\varrho^{\mathrm{triv}}_{\ul{\hamlevel}})^{-1} \cdot
    \E_{\bu \sim \pi} \E_{\bj \sim [n]}
            \bracks*{e\parens*{\sqrt{\phi};\bu, \bu^{(\bj\;n)}}}.
\]
This bound we have derived depends on the $n$th coordinate only through the transposition of~$\bj$ with~$n$.  If we repeat this derivation for a general coordinate $k \in [n]$, and then average over~$k$, we will get
\begin{equation}    \label[ineq]{eqn:marginal-bound}
    \avg_{\bk \sim [n]}\Bigl\{\text{bound on MARGINAL}_{\bk}\Bigr\} = 2\cdot \tfrac{n-1}{n} \cdot (\varrho^{\mathrm{triv}}_{\ul{\hamlevel}})^{-1} \cdot
    \E_{\substack{\bu \sim \pi \\ \mathclap{\btau \sim \transpos{n}}}}
            \bracks*{e\parens*{\sqrt{\phi};\bu, \bu^{\btau}}} = 2\cdot \tfrac{n-1}{n}  \cdot (\varrho^{\mathrm{triv}}_{\ul{\hamlevel}})^{-1} \cdot \Energy\bracks*{\sqrt{\phi}},
\end{equation}
where the $\frac{n-1}{n}$ factor accounts for the fact that when $\bj$ and $\bk$ are uniformly random, there is a $\frac{1}{n}$ chance that they are equal --- in which case the ``transposition'' $(\bj\;\bk)$  is the identity and we get a contribution of $e(\sqrt{\phi};\bu,\bu) = 0$.

\subsubsection{Bounding \texorpdfstring{CONDITIONAL${}_n$}{CONDITIONALn}}
To bound CONDITIONAL${}_n$, we again follow~\cite{LY98} to a certain point.  By definition of $\invLS{}$ we have
\[
    \E_{\ba \sim \psi_n}\bracks*{\KL{\psi_{|\ba}}{\pi_{|\ba}}} \leq \E_{\ba \sim \psi_n} \bracks*{2\invLS{\hamlevel \setminus \ba} \cdot \Energy_{\hamlevel \setminus \ba}\bracks*{\sqrt{\frac{\psi_{|\ba}}{\pi_{|\ba}}}}}.
\]
Note that  $\psi_{|\ba} = \psi(\noarg, \ba) / \psi_n(\ba)$,  similarly for~$\pi$, and also that $\Energy[c\cdot f] = c^2 \cdot \Energy[f]$.   Thus
\[
 \E_{\ba \sim \psi_n} \bracks*{2\invLS{\hamlevel \setminus \ba} \cdot \Energy_{\hamlevel \setminus \ba}\bracks*{\sqrt{\frac{\psi_{|\ba}}{\pi_{|\ba}}}} }
 = 2 \E_{\ba \sim \psi_n} \bracks*{\invLS{\hamlevel \setminus \ba} \cdot \frac{\pi_n(\ba)}{\psi_n(\ba)}\cdot \Energy_{\hamlevel \setminus \ba}\bracks*{\sqrt{\frac{\psi(\noarg, \ba)}{\pi(\noarg, \ba)}}}}
 = 2 \E_{\ba \sim \pi_n} \bracks*{\invLS{\hamlevel \setminus \ba} \cdot \Energy_{\hamlevel \setminus \ba}\bracks*{\sqrt{\frac{\psi(\noarg, \ba)}{\pi(\noarg, \ba)}}}}.
\]
We can further write
\[
    \Energy_{\hamlevel \setminus \ba}\bracks*{\sqrt{\frac{\psi(\noarg, \ba)}{\pi(\noarg, \ba)}}}
=  \Energy_{\hamlevel \setminus \ba}\bracks*{\sqrt{\phi(\noarg, \ba)}}
=  \E_{\substack{\ol{\bu} \sim \pi_{\hamlevel \setminus \ba} \\ \btau \sim \transpos{n-1}}}\bracks*{e\parens*{\sqrt{\phi}; (\ol{\bu}, \ba), (\ol{\bu}^{\btau}, \ba)}}.
\]
Combining all previous deductions, we get
\[
    \E_{\ba \sim \psi_n}\bracks*{\KL{\psi_{|\ba}}{\pi_{|\ba}}} \leq
2 \E_{\ba \sim \pi_n} \invLS{\hamlevel \setminus \ba} \cdot     \E_{\substack{\ol{\bu} \sim \pi_{\hamlevel \setminus \ba} \\ \btau \sim \transpos{n-1}}}\bracks*{e\parens*{\sqrt{\phi}; (\ol{\bu}, \ba), (\ol{\bu}^{\btau}, \ba)}} = 2\E_{\substack{\bu \sim \pi \\ \btau \sim \transpos{n-1}}}\bracks*{e\parens*{\sqrt{\phi}; \bu, \bu^{\btau}}\cdot \invLS{\hamlevel \setminus \bu_n}}.
\]
This is our desired bound for CONDITIONAL${}_n$, except we make a slight adjustment so that the expectation is over all $\btau$~in $\transpos{n}$, obtaining the equivalent bound
\[
    \text{CONDITIONAL}_n \leq 2\E_{\substack{\bu \sim \pi \\ \btau \sim \transpos{n}}}\bracks*{e\parens*{\sqrt{\phi}; \bu, \bu^{\btau}}\cdot \tfrac{\ind[n \text{ is fixed by } \btau]}{1-2/n} \cdot \invLS{\hamlevel \setminus \bu_n}}.
\]
Again, had we repeated this derivation for an arbitrary coordinate~$k$ in place of the $n$th, and then averaged over~$k$, we would get
\begin{equation}    \label[ineq]{eqn:fika1}
    \avg_{\bk \sim [n]}\Bigl\{\text{bound on CONDITIONAL}_{\bk}\Bigr\} =
    2\E_{\substack{\bu \sim \pi \\ \btau \sim \transpos{n}}}\bracks*{e\parens*{\sqrt{\phi}; \bu, \bu^{\btau}}\cdot \E_{\bk \sim \text{Fix}(\btau)} \invLS{\hamlevel \setminus \bu_{\bk}}},
\end{equation}
where $\text{Fix}(\btau)$ denotes the fixed points of transposition~$\btau$.

This is the point at which, by necessity, we depart from~\cite{LY98}.  To proceed, we simply take a worst-case upper bound on the two colors swapped by~$\btau$; no matter what $\bu$ and $\btau$ are, we have
\[
    \E_{\bk \sim \text{Fix}(\btau)} \invLS{\hamlevel \setminus \bu_{\bk}} \leq
    \max_{i_1,i_2 \in [\colors]} \braces*{\E_{\bi \sim \hamlevel \setminus \{i_1,i_2\}} \invLS{\hamlevel \setminus \bi}}.
\]
In fact, when inserting this into \Cref{eqn:fika1}, we can do slightly better.  Notice that if $\btau$ swaps two colors of $\bu$ that are the same, then $e\parens*{\sqrt{\phi}; \bu, \bu^{\btau}} = 0$ anyway.  Thus we may insert the indicator random variable $\ind[\text{$\btau$ swaps distinct colors in~$\bu$}]$ into the expectation in \Cref{eqn:fika1}, and then use
\[
    \ind[\text{$\btau$ swaps distinct colors in~$\bu$}] \cdot \E_{\bk \sim \text{Fix}(\btau)} \invLS{\hamlevel \setminus \bu_{\bk}} \leq
    \max_{\substack{i_1,i_2 \in [\colors]\\ \text{distinct}}} \braces*{\E_{\bi \sim \hamlevel \setminus \{i_1,i_2\}} \invLS{\hamlevel \setminus \bi}}.
\]
Putting this inequality into \Cref{eqn:fika1} yields
\begin{equation}    \label[ineq]{eqn:conditional-bound}
    \avg_{\bk \sim [n]}\Bigl\{\text{bound on CONDITIONAL}_{\bk}\Bigr\} \leq 2
    \max_{\substack{i_1,i_2 \in [\colors]\\ \text{distinct}}} \braces*{\E_{\bi \sim \hamlevel \setminus \{i_1,i_2\}} \invLS{\hamlevel \setminus \bi}} \cdot  \Energy\bracks*{\sqrt{\phi}}.
\end{equation}

\subsubsection{Completing the proof of \texorpdfstring{\Cref{lem:induction-structure}}{Lemma \ref{lem:induction-structure}}}
Putting together \Cref{eqn:fika0,eqn:marginal-bound,eqn:conditional-bound} yields
\[
    \KL{\phi \pi}{\pi} \leq  2 \parens*{\tfrac{n-1}{n} \cdot (\varrho^{\mathrm{triv}}_{\ul{\hamlevel}})^{-1} + \max_{\substack{i_1,i_2 \in [\colors]\\ \text{distinct}}} \braces*{\E_{\bi \sim \hamlevel \setminus \{i_1,i_2\}} \invLS{\hamlevel \setminus \bi}}} \cdot  \Energy\bracks*{\sqrt{\phi}},
\]
which immediately implies \Cref{lem:induction-structure}.

\subsection{Deducing \texorpdfstring{\Cref{thm:main}}{Theorem \ref{thm:main}} from \texorpdfstring{\Cref{lem:induction-structure}}{Lemma \ref{lem:induction-structure}}} \label{sec:finish-recursion}
In this section we upper-bound $\invLS{\hamlevel}$ using the recursion given by \Cref{lem:induction-structure}.  Let us make a few simplifications to the statement of \Cref{lem:induction-structure}.  First, let us drop the factor~$\frac{n-1}{n}$ for simplicity.  Second, let us use the upper bound on $(\varrho^{\mathrm{triv}}_{\ul{\hamlevel}})^{-1}$ from \Cref{lem:triv-bound}.  Finally, let us drop the condition that $i_1, i_2$ be distinct in the~$\max$; this condition greatly simplifies the recursion when $\colors = 2$ (as in~\cite{LY98}), but doesn't help us much when $\colors > 2$.  Thus we will finally use
\begin{equation}    \label[ineq]{eqn:recursion3}
    \invLS{\hamlevel} \leq
         \cost(\hamlevel) + \max_{\{i_1,i_2\} \subseteq \hamlevel} \braces*{\E_{\bi \sim \hamlevel \setminus \{i_1,i_2\}} \invLS{\hamlevel \setminus \bi}}, \quad \cost(\hamlevel) \coloneqq \frac12 \sum_{i\colon \hamlevel_i > 0} \lg\parens*{\frac{|\hamlevel|}{\hamlevel_i}} \quad \text{for $|\hamlevel| > 2$.}
\end{equation}
We remind the reader that in the above, $\{i_1,i_2\}$ and $\hamlevel$ are considered to be multisets of~$[\colors]$.  In fact, we will always consider $\hamlevel$ to merely be a multiset of the colors on which it is supported.  That is to say, whenever some $\hamlevel_i$ becomes~$0$ in the above recursion (through the removal of a color inside the expectation), we will treat the color~$i$ as no longer existing (rather than allowing $\hamlevel_i= 0$).  This is acceptable, since the definition of $\invLS{\hamlevel}$ is not affected by removing colors that don't appear in~$\hamlevel$.  This is why we dropped the hypothesis $\hamlevel \in \N_+^\colors$ in writing \Cref{eqn:recursion3}, and why we wrote the sum in $\cost(\hamlevel)$ as being over $\{i : \hamlevel_i > 0\}$, rather than over all $i \in [\colors]$. It might seem that dropping the hypothesis $\hamlevel \in \N_+^\colors$ (and hence $\hamlevel_i > 0$ for all~$i$) in \Cref{eqn:recursion3} could cause a problem for the case when the number of colors drops to just one, meaning  $\hamlevel = \{i, i, \dots, i\}$ for some~$i$.  However, in this degenerate case, the correct value of $\invLS{\hamlevel}$ is~$0$, and we also have~$\cost(\hamlevel) = 0$.

Regarding the base cases of $|\hamlevel| = 2$ for our recursion \Cref{eqn:recursion3}, we have the true values
\begin{equation}    \label[ineq]{eqn:base-case}
    \invLS{\{a,b\}} = \braces*{\begin{array}{lr}
                                        0 & \text{if $a = b$} \\
                                        1/2 & \text{if $a \neq b$}
                                \end{array}} \leq 1/2.
\end{equation}
Indeed, $\varrho^{\mathrm{triv}}_{\underline{\{1,2\}}} = 1$ according to \Cref{eqn:base-log-sob}, and the energy of this trivial chain is half that of the transposition chain on $\slice{\{1,2\}}$.

\paragraph{``Strategies''.}
\newcommand{\Strat}{\mathrm{Strat}}
Given $\hamlevel$, let's define a \emph{strategy for $\hamlevel$} to be a mapping~$p$ that takes in an arbitrary nonempty $\mu \subseteq \hamlevel$ and outputs a pair $\{i_1,i_2\} \subseteq \mu$.  We'll say that $p(\mu) = \{i_1,i_2\}$ is the pair \emph{protected} by~$p$.  We'll write $\Strat(\hamlevel)$ for the set of all strategies for~$\hamlevel$.  Then from \Cref{eqn:recursion3} we have that
\begin{equation}    \label[ineq]{eqn:R-vs-S}
    \invLS{\hamlevel} \leq \max_{p \in \Strat(\hamlevel)} S_p(\hamlevel),
\end{equation}
where $S_p(\hamlevel)$ is defined to be the solution of the recursion
\begin{equation}    \label{eqn:strategy1}
     S_p(\hamlevel) =
         \cost(\hamlevel) + \E_{\bi \sim \hamlevel \setminus p(\hamlevel)} S_p(\hamlevel \setminus \bi),
\end{equation}
with the base case $S_p(\text{pair}) = 1/2$ from \Cref{eqn:base-case}.  Let us write
\[
    \cost(\hamlevel) = \sum_{a \in [\colors]} \cost^{(a)}(\hamlevel), \quad \text{where } \cost^{(a)}(\hamlevel) \coloneqq \ind[\hamlevel_a \neq 0] \cdot \frac12 \lg\parens*{\frac{|\hamlevel|}{\hamlevel_a}}.
\]
Then it follows from  \Cref{eqn:strategy1} that we have
\[
    S_p(\hamlevel) = \sum_{a =1}^{\colors} S_p^{(a)},
\]
where
\begin{equation}    \label{eqn:strategy-breakup}
     S_p^{(a)}(\hamlevel) =
         \cost^{(a)}(\hamlevel) + \E_{\bi \sim \hamlevel \setminus p(\hamlevel)} S^{(a)}_p(\hamlevel \setminus \bi).
\end{equation}
Thus returning to \Cref{eqn:R-vs-S}, we have
\begin{equation}    \label[ineq]{eqn:bound-recursion1}
    \invLS{\hamlevel} \leq \max_{p \in \Strat(\hamlevel)} \braces*{\sum_{a =1}^{\colors} S_p^{(a)}(\hamlevel)} \leq \sum_{a=1}^\colors \max_{p \in \Strat(\hamlevel)} S_p^{(a)}(\hamlevel).
\end{equation}
Next, it will be convenient if $\cost^{(a)}(\hamlevel)$ is a decreasing function of $\hamlevel_a$ (considering $|\hamlevel|$ fixed), \emph{even for $\hamlevel_a = 0$}.  So let us
\begin{equation}    \label{eqn:newcost}
    \text{redefine }  \cost^{(a)}(\hamlevel) \coloneqq \frac12 \lg\parens*{\frac{2|\hamlevel|}{1+\hamlevel_a}}.
\end{equation}
This redefinition only increases $\cost^{(a)}(\hamlevel)$; it increases it from zero to a nonzero value when $\hamlevel_a = 0$; and, when $\hamlevel_a > 0$, the increase from $\hamlevel_a$ to $\hamlevel_a + 1$ in the denominator is at most a factor of~$2$, and this is compensated for by the new factor of~$2$ in the numerator. Thus \Cref{eqn:bound-recursion1} is still valid under our redefinition.

The advantage of the redefinition is, as mentioned, that $\cost^{(a)}(\hamlevel)$ always goes up when $\hamlevel_a$ drops by one, even when dropping from~$1$ down to~$0$. Thus for each fixed $a \in [\colors]$, it is now clear from from \Cref{eqn:strategy-breakup} that \emph{the optimal strategies~$p \in \Strat(\hamlevel)$ for maximizing $S_p^{(a)}(\hamlevel)$ are precisely the ``greedy'' ones}.  Here the ``greedy'' strategies for color~$a$ mean the ones that ``always protect non-$a$ colors'' (to the extent this is possible --- if $\mu$ has only $m < 2$ non-$a$ colors then $p(\mu)$ will be obliged to contain $2-m$ $a$'s).  It is also not hard to see that every such greedy strategy~$g$ is equally effective; for the purposes of computing $S_g^{(a)}(\hamlevel)$, it doesn't matter what non-$a$ colors appear in the $\mu \subseteq \hamlevel$ that arise ---  only how many of them there are. (This observation relies in part on using the same upper bound, $1/2$, for both cases in \Cref{eqn:base-case}.)

\paragraph{Greedy strategies.} Let $S_g^{(a)}(\hamlevel)$ denote the solution for a greedy protection strategy, which, as we have argued, equals $\max_{p \in \Strat(\hamlevel)} \{S_p^{(a)}(\hamlevel)\}$. Our final goal will be to establish
\begin{equation} \label[ineq]{eqn:final-goal}
    S_g^{(a)}(\hamlevel) \leq \frac12 n \lg \frac{4n}{\hamlevel_a}.
\end{equation}
Putting this bound into \Cref{eqn:bound-recursion1} will yield \Cref{thm:main}.

We first dispense with the edge case when $\hamlevel$ contains fewer than~$2$ non-$a$'s.  In this case, under the greedy strategy  $\bi$ is always~$a$ in \Cref{eqn:strategy-breakup}, and so ``solving the recursion'' just amounts to computing a sum. When $\hamlevel_a = |\hamlevel|$, the result is
\[
 S_g^{(a)} = \frac{1}{2} \lg \frac{2n}{1+n} + \cdots + \frac{1}{2} \lg \frac{2\cdot 3}{1+3} + \frac{1}{2},
\]
and when $\hamlevel_a = |\hamlevel|-1$, the result is
\[
 S_g^{(a)} = \frac{1}{2} \lg \frac{2n}{n} + \cdots + \frac{1}{2} \lg \frac{2\cdot 3}{3} + \frac{1}{2}.
\]
In both cases, each of the $n-1$ summands is at most $\frac{1}{2} \lg 2$, from which~\Cref{eqn:final-goal} immediately follows.

We now come to the main case, when $\hamlevel$ contains at least~$2$ non-$a$'s.  In this case, the greedy strategy involves always protecting two non-$a$'s, and as argued, the quantity $S_g^{(a)}(\hamlevel)$ only depends on~$\hamlevel_a$.  Thus for analysis purposes, we may henceforth assume $\colors = 2$ and $a = 1$.  Now from \Cref{eqn:bound-recursion1} and \Cref{eqn:newcost} we conclude that $S_g^{(a)}(\hamlevel)$ is   the solution of
\[
      S_g^{(1)}(\hamlevel) =
         \frac12 \lg\parens*{\frac{2|\hamlevel|}{1+\hamlevel_1}} + \frac{\hamlevel_1}{\hamlevel_1 + \hamlevel_2 - 2} S_g^{(1)}(\hamlevel \setminus 1) + \frac{\hamlevel_2 - 2}{\hamlevel_1 + \hamlevel_2 - 2} S_g^{(1)}(\hamlevel \setminus 2),
\]
with base case $S_g^{(1)}(\{2,2\}) = 1/2$.  It is easier to analyze this recursion in terms of $\lambda \coloneqq \hamlevel \setminus \{2,2\}$. Writing $G(\lambda) = S_g^{(1)}(\hamlevel)$, we have
\begin{equation}    \label{eqn:solvable}
    G(\lambda) = \frac12 \lg\parens*{\frac{2(\lambda_1 + \lambda_2 + 2)}{1+\lambda_1}}  + \frac{\lambda_1}{\lambda_1 + \lambda_2}G(\lambda \setminus 1) + \frac{\lambda_2}{\lambda_1+\lambda_2} G(\lambda \setminus 2),
\end{equation}
with base case $G(\emptyset) = 1/2$.  In fact, it will be convenient to overpay for the base case, taking $G(\emptyset) = \frac12 \lg\parens*{\frac{2(0 + 0 + 2)}{1+0}} = 1$.

Finally, we can solve the recursion in \Cref{eqn:solvable} by giving it a probabilistic interpretation.  Suppose we choose a random string $\bu$ in $\slice{\lambda}$.  Then we begin a deterministic process with ``stages'' numbered $|\lambda|, |\lambda|-1, |\lambda| - 2, \dots, 1, 0$.  In each stage, we ``pay''  $\frac12 \lg\parens*{\frac{2(\hamwt{1}{\bu} + \hamwt{2}{\bu} + 2)}{1+\hamwt{1}{\bu}}}$, and then we delete the last character in~$\bu$.  It is easy to see that the solution of \Cref{eqn:solvable} is equal to the expectation (over the initial choice of~$\bu$) of the total payment in this process.  By linearity of expectation, this total payment is the sum of the expected payment in each stage, and in the $m$th stage it is clear that the random variable $\hamwt{1}{\bu}$ is distributed as Hypergeometric($\lambda_1 + \lambda_2, \lambda_1, m$).  Thus the expected payment in the $m$th stage is
\[
    \E\bracks*{\frac12 \lg\parens*{\frac{2(m + 2)}{1+\bX}}} \leq \frac12 \lg \parens*{\E\bracks*{\frac{2(m+2)}{1+\bX}}}, \quad \text{for }  \bX \sim \text{Hypergeometric}(\lambda_1 + \lambda_2, \lambda_1, m).
\]
By \Cref{lem:inv-hypergeom}, this is at most
\[
    \frac12 \lg \parens*{\frac{2(m+2)(\lambda_1 + \lambda_2 + 1)}{(\lambda_1+1)(m+1)}} \leq  1 + \frac12 \lg \frac{\lambda_1 + \lambda_2 + 1}{\lambda_1+1}.
\]
(The bound here is a little loose, but we valued simplicity over optimization of lower-order terms.)
When this is summed over $0 \leq m \leq |\lambda|$, we get an upper bound of
\[
    |\lambda| + 1 + \frac{|\lambda|}{2}  \lg \frac{|\lambda| + 1}{\lambda_1+1}.
\]
Recalling $\lambda = \hamlevel \setminus \{2,2\}$, and using $|\hamlevel| - 1 = n-1 \leq n$, the above upper bound is at most
\[
    n + \frac12 n \lg \frac{n}{\hamlevel_1} = \frac12 n \lg \frac{4n}{\hamlevel_1},
\]
confirming \Cref{eqn:final-goal}.

\section{Applications}  \label{sec:applications}
\subsection{KKL and Kruskal--Katona for multislices}
By viewing the multislice as a Schreier graph, we can apply the results of~\cite{OW13a,OW13} to obtain the KKL Theorem in this setting (in fact, Talagrand's strengthening~\cite{Tal94} of it):
\begin{theorem} \label{thm:tal}
    Let $f\colon \slice\hamlevel \to \{0,1\}$. Then
    \[
        \avg_{\tau\in\transpos{n}} \braces*{\frac{\Inf_{\tau}[f]}{\lg(2/\Inf_{\tau}[f])}} \gtrsim \rho_{\hamlevel} \cdot \Var_{\pi_\hamlevel}[f].
    \]
\end{theorem}
(Since we have written $\gtrsim$, hiding a universal constant, it doesn't matter if we take $f$'s range to be $\{0,1\}$ or $\{-1,1\}$.)   Substituting our lower bound on $\rho_{\hamlevel}$ from \Cref{thm:main} yields concrete new results.  For example, consider our model scenario of  $n \longrightarrow \infty$ with $\colors = O(1)$ and $\hamlevel_i/n \geq \Omega(1)$ for each~$i$; suppose further that $f$ is ``roughly balanced'', meaning $\Omega(1) \leq \Var[f] \leq 1-\Omega(1)$.  Then
\[
    \avg_{\tau\in\transpos{n}} \braces*{\frac{\Inf_{\tau}[f]}{\lg(2/\Inf_{\tau}[f])}} \gtrsim \frac{1}{n}, \quad\text{and hence}\quad \calM[f] \gtrsim \frac{\log n}{n}.
\]
The latter statement here is the traditional conclusion of the KKL Theorem.

Let us record here one more concrete corollary of \Cref{thm:tal}.  In our model scenario, that theorem (roughly speaking) says that the energy $\Energy[1_A] = \avg_{\tau\in\transpos{n}} \Inf_\tau[1_A]$ is at least $\Omega\bigl(\frac{\log n}{n}\bigr)$ unless some transposition $(i\;j)$ has a rather large influence, like $1/n^{.01}$, on~$1_A$.
\begin{corollary}                                       \label{cor:kkl}
    Let $A \subseteq \slice\hamlevel$.  Assume $\kappa_i \geq pn$ for all $i \in [\colors]$ and that $\eps \leq \vol(A) \leq 1-\eps$.   Then
    \[
        \Energy[1_A] \geq \Omega\parens*{\frac{\eps}{\ell \log(1/p)}} \cdot \frac{\log(1/\calM[1_A])}{n}.
    \]
\end{corollary}
\begin{proof}
    This is immediate from \Cref{thm:tal} and \Cref{thm:main}, using $\lg(2/\Inf_\tau[1_A]) \geq \lg(2/\calM[1_A])$,  $\rho_\hamlevel^{-1} \lesssim n \cdot \colors \cdot \log(1/p)$, and $\Var[1_A] = \vol(A)(1-\vol(A)) \geq \eps/2$.
\end{proof}
Following~\cite{OW13a}, we will use this to show a variant of the Kruskal--Katona Theorem for multislices.\footnote{Our variants are unrelated to those of Clements~\cite{Cle84,Cle94,Cle98}.}

The classical Kruskal--Katona Theorem~\cite{Sch59,Kru63,Kat68} concerns subsets of Hamming slices of the Boolean cube.  To recall it, let us write a $2$-color histogram $\hamlevel \in \N_+^2$ as $(\hamlevel_0, \hamlevel_1)$, with $n = \hamlevel_0 + \hamlevel_1$. If $A \subseteq \slice\hamlevel$, then the \emph{(lower) shadow} of~$A$ is defined to be
\[
    \bdry A = \braces*{v \in \slice{(\hamlevel_0+1,\hamlevel_1-1)} : v \leq u \text{ for some } u \in A}.
\]
It is not hard to show that $\vol(\bdry A) \geq \vol(A)$ always (here the fractional volume $\vol(\bdry A)$ is vis-\`a-vis the containing slice $\slice{(\hamlevel_0+1,\hamlevel_1-1)}$).  The Kruskal--Katona Theorem improves this by giving an exactly sharp lower bound on $\vol(\bdry A)$ as a function of~$\vol(A)$.  The precise function is somewhat cumbersome to state, but the qualitative consequence, assuming that $\vol(A)$ and $\hamlevel_0/n$ are bounded away from $0$~and~$1$, is that $\vol(\bdry A) \geq \vol(A) + \Omega(1/n)$.  This is sharp, up to the constant in the $\Omega(\cdot)$, as witnessed by the ``dictator set'' $A = \{u : u_1 = 0\}$.  See~\cite[Sec.~1.2]{OW13a} for more discussion.\\

To extend the Kruskal--Katona Theorem to multislices, we first need to extend the notion of neighboring slices and shadows. Fix an ordering on the colors, $1 \prec 2 \prec \cdots \prec \colors$. This total order extends to a partial order on strings in $[\colors]^n$ in the natural way.
\begin{definition}  \label{def:shadows}
Let $\hamlevel \in \N_+^\colors$ be a histogram. We say that histogram $\hamlevel'$ is a \emph{lower neighbor} of $\hamlevel$, and write $\hamlevel' \triangleleft \hamlevel$, if there exists some $c \prec d \in [\colors]$ such that $\hamlevel_c' = \hamlevel^{\vphantom{'}}_d + 1$, $\hamlevel'_d = \hamlevel^{\vphantom{'}}_c - 1$, and $\hamlevel'_i = \hamlevel^{\vphantom{'}}_i$ for all other colors~$i$.  In the opposite case, when $c \succ d$, we say $\hamlevel'$ is an \emph{upper neighbor} of~$\hamlevel$, and write $\hamlevel' \triangleright \hamlevel$.
\end{definition}

The main difference between the Boolean case and the multicolored case is that each multislice now has multiple upper and lower neighbors.
\begin{definition}
Let $A \subseteq \slice\hamlevel$, and let $\hamlevel' \triangleleft \hamlevel$. The \emph{lower shadow of $A$ at $\hamlevel'$} is
\[
    \bdry_{\hamlevel'}A = \{u \in \slice{\hamlevel'} : u \prec v \text{ for some } v \in A\}.
\]
We similarly define upper shadows.  We may use the same notation $\bdry_{\hamlevel'} A$ for both kinds of shadows, since whether a shadow is upper or lower is determined by whether $\hamlevel' \triangleright \hamlevel$ or $\hamlevel' \triangleleft \hamlevel$.
\end{definition}

Towards proving Kruskal--Katona theorems, we relate the volume of $A$'s lower shadows to~$\Energy[1_A]$.  Recalling that $A$ now has multiple lower shadows, we show that a certain weighted average of their volumes is noticeably larger than the volume of~$A$.  We first define the appropriate weighted average.
\begin{definition}
    Given a histogram $\hamlevel \in \N_+^\colors$, we define a natural probability distribution $\lwr(\hamlevel)$ on the lower neighbors of~$\hamlevel$ as follows.  To draw $\hamlevel' \sim \lwr(\hamlevel)$: take an arbitrary $u \in \slice\hamlevel$; choose $\bj, \bj' \sim [n]$ independently and randomly, conditioned on $u_{\bj} \neq u_{\bj'}$; let $\bc, \bd$ denote the two colors $u_{\bj}, u_{\bj'}$, with the convention $\bc \prec \bd$; finally, let $\hamlevel'$ be the lower neighbor of $\hamlevel$ with $\hamlevel'_{\bc} = \hamlevel_{\bc} + 1$ and $\hamlevel'_{\bd} = \hamlevel_{\bd} - 1$.

    We similarly define a probability distribution $\upr(\hamlevel)$ on the upper neighbors of $\hamlevel$ by interchanging the roles of $\bc$ and~$\bd$.
\end{definition}

\begin{proposition}                                     \label{prop:energy-vs-shadow}
    Given a histogram $\hamlevel \in \N_+^\colors$, let
    $
        h(\hamlevel) = 1 - \sum_{i=1}^\colors \frac{\hamlevel_i(\hamlevel_i-1)}{n(n-1)} \leq 1.
    $
    (This is the probability that applying a random transposition to a string in $\slice\hamlevel$ actually changes it.)
    Then for any $A \subseteq \slice\hamlevel$,
    \[
        \E_{\bkappa' \sim \lwr(\hamlevel)}\bracks*{\vol(\bdry_{\bkappa'} A)} \geq \vol(A) + \Energy[1_A]/h(\hamlevel).
    \]
    In particular, at least one lower neighbor of~$A$ has volume at least $\vol(A) + \Energy[1_A]/h(\hamlevel)$.
\end{proposition}
\begin{remark}
    The same proposition also holds if we consider upper neighbors, $\bkappa' \sim \upr(\hamlevel)$.
\end{remark}
\begin{proof}
    By definition,
    \begin{equation}    \label{eqn:en1a}
        \Energy[1_A] = \Pr_{\substack{\bu \sim \pi_\hamlevel \\ \btau \sim \transpos{n}}}[\bu \in A,\ \bu^{\btau} \not \in A] = h(\hamlevel) \cdot \Pr[\bu \in A,\ \bv \not \in A],
    \end{equation}
    where the random string $\bv \in \slice\hamlevel$ is defined to be $\bu^{\btau}$ conditioned on $\bu \neq \bu^{\tau}$.  In other words, the pair $(\bu,\bv)$ is distributed as a random pair of strings differing by a ``nontrival'' color-swap.  Let $\bc, \bd \in [\colors]$ denote the two colors swapped, with the convention $\bc \prec \bd$.  Then if we define $\hamlevel_{\bc\bd}\triangleleft \hamlevel$ to be the lower neighbor of~$\hamlevel$ having one fewer~$\bd$ and one more~$\bc$, it holds that $\hamlevel_{\bc\bd}$ is distributed according to $\lwr(\hamlevel)$.  Finally, let $\bw \in \slice{\hamlevel_{\bc\bd}}$ to be the string that agrees with $\bu, \bv$ on the unswapped coordinates, and has color~$\bc$ on the swapped coordinates.

    It is easy to see the following: $\bv$ is uniformly distributed on $\slice{\hamlevel}$; conditioned on $\bc$ and $\bd$, the string $\bw$ is uniformly distributed on $\slice{\hamlevel_{\bc\bd}}$; and, $\bw \prec \bu$, $\bw \prec \bv$.  In light of the last of these, we may make the following deductions: If $\bv \in A$, then $\bw \in \bdry_{\hamlevel_{\bc\bd}}A$.  Furthermore, even when $\bv \not \in A$, if $\bu \in A$ then $\bw \in \bdry_{\hamlevel_{\bc\bd}}A$.  Thus
    \begin{align*}
        \Pr[\bw \in \bdry_{\hamlevel_{\bc\bd}}A] &\geq \Pr[\bv \in A] + \Pr[\bu \in A,\ \bv \not \in A] \nonumber \\
        \implies \quad \E_{\bc,\bd}[\vol(\bdry_{\hamlevel_{\bc\bd}}A)] &\geq \vol(A) + \Energy[1_A]/h(\hamlevel). \label{eqn:avg}
    \end{align*}
    In this deduction, on the right we used that $\bv$ is uniformly distributed on~$\slice\hamlevel$ and we used \Cref{eqn:en1a}. On the left we used that --- conditioned on $\bc$ and $\bd$ --- the string $\bw$ is uniform on $\slice{\hamlevel_{\bc\bd}}$.  The proof is completed by recalling that $\kappa_{\bc\bd}$ is distributed according to $\lwr(\hamlevel)$.
\end{proof}

We can now immediately deduce our first Kruskal--Katona Theorem, using just the log-Sobolev inequality \Cref{thm:main}, and \Cref{eqn:sse1}:
\begin{theorem}                                     \label{thm:kk1}
    For $A \subseteq \slice\hamlevel$ we have
    \[
        \E_{\bkappa' \sim \lwr(\hamlevel)}\bracks*{\vol(\bdry_{\bkappa'} A)} \geq \vol(A) + \frac{1}{n} \cdot \vol(A) \ln(1/\vol(A)) \cdot \parens*{\littlesum_{i=1}^n \log_2(4n/\hamlevel_i)}^{-1}.
    \]
    In particular, at least one lower shadow of~$A$ has volume at least the right-hand side.  The analogous statement for upper shadows also holds.
\end{theorem}
Thus in the model case when $\vol(A)$ and each $\hamlevel_i/n$ is bounded away from $0$ and~$1$, and $\colors = O(1)$, we get that the average lower shadow of~$A$ has volume at least $\vol(A) + \Omega(1/n)$.  Now using our KKL Theorem (\Cref{cor:kkl}) we can get a ``robust'' version of this statement; the volume increase is in fact on the order of $(\log n)/n$ unless there is a highly influential transposition for~$A$:
\begin{theorem}                                       \label{thm:kkl-robust}
    Let $A \subseteq \slice\hamlevel$.  Assume $\kappa_i \geq pn$ for all $i \in [\colors]$ and that $\eps \leq \vol(A) \leq 1-\eps$.   Then for any $\delta > 0$ we have
    \[
        \E_{\bkappa' \sim \lwr(\hamlevel)}\bracks*{\vol(\bdry_{\bkappa'} A)} \geq \vol(A) + \frac{\log n}{n} \cdot \Omega\parens*{\frac{\eps \delta}{\ell \log(1/p)}},
    \]
    or else there exists $\tau \in \transpos{n}$ with $\Inf_{\tau}[A] \geq 1/n^{\delta}$. The analogous statement for upper shadows also holds.
\end{theorem}

\ignore{
Define $\nu(\hamlevel) = \frac{\colors^2(\colors-1)^2}{2n(n-1)} (\max_i \hamlevel_i)^2$.\ynote{I changed the first factor from $\frac{\colors(\colors-1)}{n(n-1)}$.} Note that if $\colors$ is a constant and $\min_i \hamlevel_i$ is at least a constant fraction of $n$, then $\nu(\hamlevel)$ is also a constant.
\begin{proposition}
Let $A \subset \slice\hamlevel$. Then there exists some upper neighbor $\hamlevel'$ of $\hamlevel$ such that
\[
	\mu_{\hamlevel'}(\bdry_{\hamlevel'} A) \geq \mu_\hamlevel + \Energy[1_A]/\nu(\hamlevel).
\]
\end{proposition}
\begin{proof}
First fix two colors $c \prec d \in [\colors]$.
Let $\bu\sim \slice\hamlevel$ be uniformly random. Now let $\bi \sim [n]$ be a uniformly random index conditioned on $\bu_{\bi} = c$, and let $\bj \sim [n]$ be a uniformly random index conditioned on $\bu_\bj = d$. Let $\bv = \bu^{(\bi\ \bj)}$ and $\bw = \bu^{\bi \to d}$ (i.e., $\bu$ with its $\bi^\text{th}$ coordinate changed to $d$). Let $\hamlevel'$ denote the upper-neighboring slice that $\bw$ is on. Then $\bv$ is also uniformly distributed on $\slice\hamlevel$, and $\bw$ is uniformly distributed on $\slice{\hamlevel'}$. Also, if either $\bu \in A$ or $\bv \in A$, then $\bw \in \bdry_{\hamlevel'} A$. Then we can write
\begin{align}\label[ineq]{eqn:upper-shadow-size}
\mu_{\hamlevel'}(\bdry_{\hamlevel'} A) &= \Pr[\bw \in \bdry_{\hamlevel'}A] \geq \Pr[\bu \in A \text{ or } \bv \in A] = \Pr[\bu \in A] + \Pr[\bu\notin A \text{ and }\bv \in A].
\end{align}
Consider the second event above. By symmetry,
\begin{align*}
 \Pr[\bu\notin A \text{ and }\bv \in A] &= \frac12 \Pr[\bone_A(\bu) \ne \bone_A(\bv)].
\end{align*}
Now consider choosing $\bu' \sim \slice\hamlevel$ uniformly, and also a transposition $\tau = (\bi' \ \bj') \sim \transpos{n}$ uniformly. Let $\bv' = (\bu')^\tau$. Let $E_{c,d}$ be the event that $\{\bu'_{\bi'}, \bu'_{\bj'}\} = \{c,d\}$. Conditioned on $E_{c,d}$, $(\bu',\bv')$ has the same distribution as $(\bu,\bv)$. Therefore, combining with \Cref{eqn:upper-shadow-size}, we have
\begin{align*}
\mu_{\hamlevel'}(\bdry_{\hamlevel'} A) \geq \Pr[\bu\in A] + \frac12 \Pr[\bone_A(\bu') \ne \bone_A(\bv')\,|\,E_{c,d}].
\end{align*}
Now, we sum over all choices of $c \prec d$, and get
\begin{align*}
\sum_{\hamlevel \trianglelefteq \hamlevel'} \mu_{\hamlevel'}(\bdry_{\hamlevel'} A) &\geq \binom{\colors}{2}\Pr[\bu\in A] + \sum_{c \prec d} \frac12 \Pr[\bone_A(\bu') \ne \bone_A(\bv')\,|\,E_{c,d}]\\
&= \binom{\colors}{2}\Pr[\bu\in A] + \frac12 \Pr[\bone_A(\bu') \ne \bone_A(\bv')\,|\,\bu' \ne \bv']\\
&= \binom{\colors}{2}\Pr[\bu\in A] + \frac12 \Pr[\bone_A(\bu') \ne \bone_A(\bv')]\cdot(1/\Pr[\bu' \ne \bv']).
\end{align*}
Meanwhile,
\[
\Pr[\bu' \ne \bv'] = \sum_{c \prec d \in [\colors]} \Pr[\{\bu'_{\bi'}, \bu'_{\bj'}\} = \{c,d\}] = \sum_{c < d \in [\colors]} \frac{(\hamwt{c}{\bu'})(\hamwt{d}{\bu'})}{\binom n 2} \leq \frac{\binom{\colors}{2}}{\binom{n}{2}} \bigl(\max_i \hamlevel_i\bigr)^2.
\]
Therefore,
\[
\max_{\hamlevel \trianglelefteq \hamlevel'} \mu_{\hamlevel'}(\bdry_{\hamlevel'} A) \geq \Pr[\bu\in A] +
\frac{\Energy[\bone_A]}{\binom{\colors}{2}} \cdot \frac{\binom{n}{2}}{\binom{\colors}{2}} \bigl( \max_i \hamlevel_i \bigr)^{-2}. \qedhere
\]
\end{proof}

Hence applying generalized KKL, we get the following corollary:
\begin{corollary}
For all $\ep > 0$ and constant $\colors$ there exists $\delta > 0$ such that the following holds: If $A \subseteq \slice\hamlevel$, $\ep \leq \min_i\hamlevel_i$, and $\ep \leq \mu_\hamlevel(A) \leq 1-\ep$, then there exists an upper neighbor $\hamlevel'$ of $\hamlevel$ such that
\begin{equation}
\mu_{\hamlevel'}(\bdry_{\hamlevel'} A) \geq \mu_{\hamlevel}(A) + \delta \cdot \frac{\log n}{n},
\end{equation}
unless there exists $\tau \in \transpos n$ such that $\Inf_\tau(\bone_A) \geq \frac 1 {n^\ep}$.
\end{corollary}

}

As in~\cite{OW13a}, we now give a conceptual improvement to the ``or else'' clause in \Cref{thm:kkl-robust}.  Let us work with upper shadows rather than lower shadows going forward.  The natural example for sets~$A$ with upper-shadow expansion ``only'' $\Omega(1/n)$ are ``dictator'' sets such as $A = \{u : u_1 = \ell\}$.  For such sets, all transpositions of the form $(1\;j)$ indeed have huge influence. However, it's not so natural to single out one such~$(1\;j)$ as the ``reason'' for the small expansion; instead, we would prefer to say the reason is that~$A$ is highly ``correlated'' with coordinate~$1$.  To this end, let us make a definition.
\begin{definition}
Let $A \subseteq \slice\hamlevel$,  let $j \in [n]$, and let $c \prec d$ be colors in~$[\colors]$.  The \emph{correlation} of $A$ with respect to coordinate $j$ and colors $c,d$ is
\[
\corr_{j,c,d}[A] = \Pr_{\bu \sim \pi_\hamlevel}[\bu \in A \mid \bu_j = d] - \Pr_{\bu \sim \pi_\hamlevel} [\bu \in A \mid \bu_j = c].
\]
\end{definition}

For simplicity, we present the following theorem without stating the most general possible settings for parameters:
\begin{theorem}                                     \label{thm:kkl-robust-better-simple}
    For $n \longrightarrow \infty$, let $A \subseteq \slice\hamlevel$, with $\colors = O(1)$,  $\kappa_i/n \geq \Omega(1)$ for all $i \in [\colors]$ and $\Omega(1) \leq \vol(A) \leq 1 - \Omega(1)$.   Then
    \begin{equation} \label[ineq]{eqn:robust-kk-hypothesis}
        \E_{\kappa' \sim \upr(\hamlevel)} \bracks*{\vol(\bdry_{\kappa'} A)} \geq \vol(A) + \Omega\parens*{\frac{\log n}{n}},
    \end{equation}
    or else there exists $j \in [n]$ and colors $c \prec d \in [\ell]$ with $\corr_{j,c,d}[A] \geq 1/n^{.01}$.
\end{theorem}

Suppose that \Cref{eqn:robust-kk-hypothesis} doesn't hold. Then there must exist $\tau \in \transpos n$ such that $\Inf_\tau[A] \geq 1/n^{.01}$. Without loss of generality, we can assume that $\tau = (1\;2)$. We then deduce two consequences of this, relating the volume of $A$ and its upper shadows. Finally we will combine these to get that $A$ is correlated to a single color change on one coordinate. The proof here has similar ideas to~\cite[Lemma A.6]{OW13a}, but we reproduce it for completeness. We also introduce the notation $\vol_c(A) \coloneqq \Pr_{\bu \sim \pi_\kappa} [\bu \in A ~|~ \bu_1 = c]$.

\begin{lemma}\label{lem:robust-kk-1}
Let $A \subseteq \slice\hamlevel$, with $\colors = O(1)$ as $n \to \infty$. Let $\eps \leq \hamlevel_c /n \leq 1-\eps$ for all $c \in [\ell]$. Further suppose that
\[
	\E_{\hamlevel' \sim \upr(\hamlevel)} \bracks*{\vol(\bdry_{\hamlevel'} A)} - \vol(A) \leq \eta.
\]
For every $c,d \in [\ell]$ with $c \prec d$, let $\hamlevel_{dc} \triangleright \hamlevel$ be the upper neighbor with one more $d$ and one less $c$ than $\hamlevel$. Then we have
\begin{equation} \label{eqn:robust-kk-interm}
	\sum_{e \in \{c,d\}}  \parens*{\Pr_{\bv \sim \pi_{\hamlevel_{dc}}}\bracks*{\bv \in \bdry_{\hamlevel_{dc}}A \land \bv_1 = e} - \Pr_{\bu \sim \pi_{\hamlevel}}\bracks*{\bu \in A \land \bu_1 = e}} \leq \frac{1}{\ep^2} \eta.
\end{equation}
\end{lemma}
\begin{proof} Let $P_{\bc,\bd}$ be the probability that $\kappa' \sim \upr(\kappa)$ is such that $\kappa'$ has one more $\bd$ and one less $\bc$ than $\kappa$. We use the convention that $\bu \sim \slice\hamlevel$ and $\bv$ is sampled from an upper neighbor of $\bu$. By definition,
\begin{align*}
    \E_{\hamlevel' \sim \upr(\hamlevel)} \bracks*{\vol(\bdry_{\hamlevel'} A)} - \vol(A) &= \sum_{\bc,\bd,e \in [\ell]} \vol_e (\bdry_{\kappa_{\bd\bc}}A) \Pr[\bv_1 = e] P_{\bc,\bd} - \sum_{\bc,\bd,e \in [\ell]} \vol_e(A) \Pr[\bu_1 = e] P_{\bc,\bd}.
\end{align*}
If we sample $\bu \sim \slice\hamlevel$ conditioned on $\bu_1 = e$, and change a random $\bc$ to a $\bd$, this distribution is uniform on $\slice{\hamlevel_{\bd\bc}}$ conditioned on the first coordinate being $e$. Therefore, conditioned on $\bc$ and $\bd$, for $e \in [\ell]$, $\vol_e(A) \leq \vol_e(\bdry_{\kappa_{\bd\bc}} A)$. We also have that for $e \ne \bc, \bd$, $\Pr[\bv_1 = e] = \Pr[\bu_1 = e]$. Putting these together,
\[
    \sum_{\bc,\bd \in [\ell]}\sum_{e \in \{\bc,\bd\}} \parens*{\vol_e (\bdry_{\kappa_{\bd\bc}}A) \Pr[\bv_1 = e] P_{\bc,\bd} - \vol_e(A) \Pr[\bu_1 = e] P_{\bc,\bd}} \leq \eta.
\]
In particular, every pair $c,d \in [\colors]$ satisfies
\[
	\sum_{e \in \{c,d\}}  \parens*{\Pr_{\bv \sim \pi_{\hamlevel_{dc}}}\bracks*{\bv \in \bdry_{\hamlevel_{dc}}A \land \bv_1 = e} - \Pr_{\bu \sim \pi_{\hamlevel}}\bracks*{\bu \in A \land \bu_1 = e}}P_{c,d} \leq \eta.
\]
Finally, we bound $P_{c,d}$. This is the probability that, for any $u \in \slice\hamlevel$ and $\bi,\bj \in [n]$ chosen uniformly and independently, $\{u_{\bi}, u_{\bj}\} = \{c,d\}$ conditioned on $u_{\bi} \ne u_{\bj}$. We can calculate this probability explicitly:
\begin{align*}
\Pr_{\bi,\bj \sim [n]} \bracks*{u_{\bi} = c, u_{\bj} = d \mid u_{\bi} \ne u_{\bj}} &= \frac{\kappa_c}{n} \cdot \frac{\kappa_d}{n-\kappa_c} \cdot h(\kappa)^{-1} \geq \eps^2.
\end{align*}
Therefore,
\[
	\sum_{e \in \{c,d\}}  \parens*{\Pr_{\bv \sim \pi_{\hamlevel_{dc}}}\bracks*{\bv \in \bdry_{\hamlevel_{dc}}A \land \bv_1 = e} - \Pr_{\bu \sim \pi_{\hamlevel}}\bracks*{\bu \in A \land \bu_1 = d}} \leq \frac{1}{\eps^2}\eta. \qedhere
\]
\end{proof}

\begin{lemma}\label{lem:robust-kk-2}
Let $A \subseteq \slice\hamlevel$, with $\colors = O(1)$ as $n \longrightarrow \infty$. Let $\eps \leq \hamlevel_c /n \leq 1-\eps$ for all $c \in [\ell]$. Further suppose that
\[
    \Inf_{(1\;2)}[A] \geq \colors^2 \gamma.
\]
Then there exist $c,d \in [\colors]$ with $c\prec d$ such that the upper neighbor $\hamlevel_{dc} \triangleright \hamlevel$ satisfies
\begin{equation} \label{eqn:robust-kk-interm2}
	\Pr_{\bv \sim \pi_{\hamlevel_{dc}}}\bracks*{\bv \in \bdry_{\hamlevel_{dc}}A~|~\bv_i = d} - \Pr_{\bu \sim \pi_{\hamlevel}}\bracks*{\bu\in A~|~\bu_i = c} \geq \gamma
\end{equation}
for $i=1$ or $2$.
\end{lemma}
\begin{proof}
Draw $\bu \sim \pi_\kappa$, and write $\bu = (\bu_1, \bu') = (\bu_1,\bu_2,\bw)$ with $\bw \in [\colors]^{n-2}$. We have
\[
\Inf_{(1\;2)}[A] \geq \Pr[(\bu_1,\bu_2,\bw)\notin A \land (\bu_2,\bu_1,\bw) \in A] \geq \colors^2\gamma.
\]
Since there are $\colors^2$ choices of colors for $\bu_1$ and $\bu_2$, it must be true that for some choice of $c,d \in [\colors]$,
\[
\Pr[\bu_1 = c \land \bu_2 = d \land (c,d,\bw) \notin A \land (d,c,\bw) \in A] \geq \gamma.
\]
We consider the case that $c \prec d$ and reach the case of $i=1$ in the conclusion; the other case is similar.

If $(d,c,\bw) \in A$ then $(d,d,\bw) \in \bdry_{\kappa_{dc}} A$.
Therefore
\[
\Pr[\bu_1 = c \land \bu_2 = d \land (c,d,\bw) \notin A \land (d,d,\bw) \in \bdry_{\kappa_{dc}}A] \geq \gamma.
\]
Every $\bu$ in the event above also satisfies $\bu \notin A \land (d,\bu') \in \bdry_{\kappa_{dc}}A$, so
\[
\Pr[\bu_1 = c \land \bu \notin A \land (d,\bu') \in \bdry_{\kappa_{dc}}A] \geq \gamma,
\]
and thus
\[
\Pr[\bu \notin A \land (d,\bu') \in \bdry_{\kappa_{dc}}A ~|~ \bu_1 = c] \geq \frac{\gamma}{\kappa_c/n} \geq \gamma.
\]

Finally, let $\bu' \sim \pi_{\kappa'}$, where $\kappa'$ is obtained from $\kappa$ by removing one $c$.
If $(c,\bu') \in A$ then $(d,\bu') \in \bdry_{\kappa_{dc}}A$, and so
\[
\Pr_{\bv' \sim \pi_{\kappa'}}\bracks*{(d,\bv') \in \bdry_{\kappa_{dc}}A} - \Pr_{\bu' \sim \pi_{\kappa'}}\bracks*{(c,\bu')\in A} = \Pr_{\bu' \sim \pi_{\kappa'}} \bracks*{(d,\bu') \in \bdry_{\kappa_{dc}} A \land (c,\bu') \notin A}.
\]
Now $(c,\bu')$ is uniformly distributed on $\slice{\hamlevel}$ conditioned on $\bu_1 = c$, so
\begin{align*}
\Pr_{\bu' \sim \pi_{\kappa'}} \bracks*{(d,\bu') \in \bdry_{\kappa_{dc}} A \land (c,\bu') \notin A} = \Pr_{\bu \sim \pi_{\kappa}} \bracks*{\bu \notin A \land (d,\bu') \in \bdry_{\kappa_{dc}} A ~|~ \bu_1 = c} \geq \gamma,
\end{align*}
and also
\[
\Pr_{\bu' \sim \pi_{\kappa'}}\bracks*{(c,\bu') \in A} = \Pr_{\bu \sim \pi_{\hamlevel}}\bracks*{\bu\in A~|~\bu_i = c}.
\]
Meanwhile, $(d,\bu')$ is uniformly distributed on $\slice{\hamlevel_{dc}}$ conditioned on $\bu_1 = d$, and so
\[
\Pr_{\bv' \sim \pi_{\kappa'}}\bracks*{(d,\bv') \in \bdry_{\kappa_{dc}}A} = \Pr_{\bv \sim \pi_{\hamlevel_{dc}}}\bracks*{\bv \in \bdry_{\hamlevel_{dc}}A~|~\bv_i = d}. \qedhere
\]
\end{proof}

\begin{proof}[Proof of~\Cref{thm:kkl-robust-better-simple}]
Apply \Cref{thm:kkl-robust} with $\delta = .02$, and assume that \Cref{eqn:robust-kk-hypothesis} does not hold. The theorem shows that $\Inf_\tau[A] \geq 1/n^{.02}$ for some transposition $\tau$, which without loss of generality is $\tau = (1\;2)$. We can therefore apply \Cref{lem:robust-kk-2} (with $\gamma = 1/(\ell^2 n^{.02})$), obtaining two colors $c \prec d$, which without loss of generality satisfy the conclusion of the lemma for $i=1$:
\begin{equation} \label[ineq]{eqn:robust-kk-2-cons}
\vol_d(\bdry_{\kappa_{dc}}A)\geq \frac{1}{\ell^2 n^{.02}} + \vol_c(A).
\end{equation}
We will show that $A$ is correlated to the first coordinate and the colors $c$ and $d$.

Since \Cref{eqn:robust-kk-hypothesis} does not hold, we can apply \Cref{lem:robust-kk-1} (with $\eta = O(\log n/n)$) to obtain
\[
	\Pr[\bv_1 = d]\vol_d(\bdry_{\kappa_{dc}} A) +\Pr[\bv_1 = c]\vol_c(\bdry_{\kappa_{dc}} A) - \Pr[\bu_1 = d]\vol_d(A) - \Pr[\bu_1 = c]\vol_c(A) \leq \frac{1}{\eps^2} \eta.
\]
Combining \Cref{eqn:robust-kk-2-cons} with $\Pr[\bv_1 = d] \geq \ep$ and $\vol_c(\bdry_{\kappa_{dc}}A) \geq \vol_c(A)$, we deduce
\begin{align*}
\Pr[\bu_1 = d]\vol_d(A) &\geq \Pr[\bv_1 = d]\vol_d(\bdry_{\kappa_{dc}} A) + \Pr[\bv_1 = c]\vol_c(\bdry_{\kappa_{dc}} A) - \Pr[\bu_1 = c]\vol_c(A) - \frac{1}{\eps^2} \eta\\
&\geq \frac{\eps}{\colors^2 n^{.02}} + (\Pr[\bv_1 = d] + \Pr[\bv_1 = c] - \Pr[\bu_1 = c])\vol_c(A) - \frac{1}{\eps^2} \eta \\
&\geq \Pr[\bu_1 = d]\vol_c(A) + \frac{\eps}{\colors^2 n^{.02}} - \frac{1}{\eps^2} \eta.
\end{align*}
Dividing by $\Pr[\bu_1 = d]$, we conclude that
\[
\corr_{1,c,d}[A] = \vol_d(A) - \vol_c(A) \geq \frac{\eps}{\colors^2 n^{.02}} - \frac{1}{\eps^2} \eta.
\]
Since $\eta = O(\log n/n)$, $\ell = O(1)$, and $\eps = \Omega(1)$, for large enough $n$, $\corr_{1,c,d}[A] \geq \frac{1}{n^{.01}}$.
\end{proof}

\ignore{
XXXXXXXXXXXXOld version of this Theorem is belowXXXXXXXXXXX

\begin{proposition}
Let $A \subseteq \slice\hamlevel$ and let $\eta,\gamma > 0$ such that
\begin{equation}\label[ineq]{eqn:corr-kk-hypothesis}
    \mu_{\hamlevel'} (\bdry_{\hamlevel'} A) - \mu_\hamlevel (A) \leq \eta
\end{equation}
for every upper neighbor $\hamlevel'$ of $\hamlevel$, and
\[
    \Inf_{(i\ j)} (\bone_A) \geq \colors^2\gamma.
\]
Then there exist $c_1 \prec c_2 \in [\colors]$ such that
\[
\max(\corr_{i,c_1,c_2}(A),\corr_{j,c_1,c_2}(A)) \geq
\frac{\gamma}{\hamlevel_{c_1}} - \frac{\eta}{\hamlevel_{c_2}}.
\]
\end{proposition}
\begin{proof}
We follow the proof of Proposition~3.6 in~\cite{OW13a}.

Assume without loss of generality that $i = 1$ and $j = 2$. Draw $\bu \sim \slice\hamlevel$ uniformly, and write $\bu = (\bu_1,\bu_2,\bw)$ with $w \in [\colors]^{n-2}$. We have
\[
\Inf_{(1\;2)}[\bone_A] = \Pr[\bu_1 \ne \bu_2 \land \bone_A(\bu_1,\bu_2,\bw) \ne \bone_A(\bu_2,\bu_1,\bw)] \geq \colors^2\gamma,
\]
and we can partition the event on the right as disjoint events of the form
\[
\bu_1 \ne \bu_2 \land (c_1,c_2,\bw) \notin A \land (c_2,c_1,\bw)\in A
\]
for every choice of $c_1 \ne c_2 \in [\colors]$. One of these must occur with probability at least $2\gamma$. If $c_1 < c_2$, we lower bound $\corr_{1,c_1,c_2}(A)$, otherwise we lower bound $\corr_{2,c_2,c_1}(A)$.

Let $\hamlevel'$ be $\hamlevel$ with one $c_1$ changed to $c_2$. By hypothesis \Cref{eqn:corr-kk-hypothesis}, if we draw $\bv \sim \slice{\hamlevel'}$ and $\bu \sim \slice\hamlevel$, and let $\bv_1$ and $\bu_1$ be the marginals on the first coordinate respectively, then
\[
\sum_{c \in [\colors]} \Pr[\bv_1 = c] \Pr[\bv \in \bdry_{\hamlevel'}A\,|\,\bv_1=c] - \sum_{c \in [\colors]} \Pr[\bu_1 = c]\Pr[\bu \in A \,|\, \bu_1 = c] \leq \eta.
\]
For all $c \in [\colors]\setminus \{c_1,c_2\}$,
\[
\Pr[\bu \in A \,|\, \bu_1 = c] \leq \Pr[\bv \in \bdry_{\hamlevel'}A \,|\, \bv_1 = c],
\]
since if we sample $(\bu|\bu_1 = c)$ from $\slice\hamlevel$ and then choose a random coordinate where $\bu$ is $c_1$ and change it to a $c_2$, we get the distribution $(\bv|\bv_1 = c)$. Furthermore, $\Pr[\bu_1 = c] = \Pr[\bv_1 = c]$. Therefore,
\[
\sum_{c \in \{c_1,c_2\}} \Pr[\bv_1 = c] \Pr[\bv \in \bdry_{\hamlevel'}A\,|\,\bv_1=c] - \sum_{c \in \{c_1,c_2\}} \Pr[\bu_1 = c]\Pr[\bu \in A \,|\, \bu_1 = c] \leq \eta.
\]
Now the proof continues exactly as in~\cite{OW13a}, using $c_1$ and $c_2$ in place of $0$ and $1$, respectively.
\end{proof}}

\subsection{Harmonic analysis on the symmetric group, and Friedgut on the multislice} \label{sec:constant-degree}
In this section we will recap some aspects of harmonic analysis on the symmetric group and on the multislice, paying particular attention to the notion of the ``low-degree'' components of a function.  For more details, see e.g.~\cite{Dia88}.

First, we briefly discuss partitions. A \emph{partition} $\lambda$ of $n$ is a nonincreasing sequence of positive integers summing to $n$. (Equivalently, it is a sorted histogram $\kappa$; i.e., one with $\kappa_1 \geq \kappa_2 \geq \cdots \geq \kappa_\colors$.)  We write $\lambda \vdash n$.  Sometimes we extend $\lambda$ into an infinite sequence, by padding it with infinitely many zeroes. We say that $\lambda$ \emph{dominates} or \emph{majorizes} $\mu$, written $\lambda \unrhd \mu$, if for all $i \geq 1$ the inequality $\lambda_1 + \cdots + \lambda_i \geq \mu_1 + \cdots + \mu_i$ holds.

Though we will eventually be interested in functions on multislices, we begin by studying the larger vector space~$V$ of functions $f\colon \symm{n} \to \R$ on the symmetric group.  Note that we can naturally extend the operators $\Ker$, $\Lap$, $\Heat_t$ to this space~$V$.  The partitions $\lambda$ of $n$ index the \emph{irreducible representations} of the symmetric group~$\symm{n}$.  In particular this means that $V$ has an orthogonal decomposition
\[
    V = \{f \mid f \colon \symm{n} \to \R\} = \bigoplus_{\lambda \vdash n} V^\lambda,
\]
where the \emph{isotypic component} $V^\lambda$ corresponds to the irreducible representations~$\lambda$ (counted with multiplicity). In analogy with the level/degree decomposition on the Boolean cube, we denote the orthogonal projection of $f$ onto $V^\lambda$ by $f^{=\lambda}$.

One utility of this decomposition is that $V^\lambda$ is an eigenspace for the operator~$\Ker$, with eigenvalue equal to $\wh{\chi}_\lambda(\tau)$, the normalized character evaluated at a(ny) transposition $\tau \in \transpos{n}$.   Frobenius~\cite{Fro00} determined an explicit formula for these character values:
\begin{equation} \label{eq:frobenius-ker}
 \Ker f
 = \sum_{\lambda \vdash n} c_\lambda f^{=\lambda}, \text{ where } c_\lambda = \frac{1}{n(n-1)}\sum_{i=1}^\colors [\lambda_i^2 - (2i-1)\lambda_i].
\end{equation}
See \cite[Cor.~1 \& Lem.~7]{DS81} for an explicit proof of the above.  Immediate consequences of the above formula are the following:
\begin{align} \label{eq:frobenius-lap}
 \Lap f &= \sum_\lambda d_\lambda f^{=\lambda}, \text{ where } d_\lambda = 1 - c_\lambda; \\
\label{eq:frobenius-Ht}
 \Heat_tf &= e^{-t \Lap}f = \sum_\lambda e^{-td_\lambda} f^{=\lambda}.
\end{align}
From \Cref{eq:frobenius-Ht} we can see that $\Heat_t$ is an invertible operator for all $t \geq 0$, and that it is natural to write $\Heat_{t}^{-1} = \Heat_{-t}$.

An important feature of the formula for $c_\lambda$ (and hence $d_\lambda$) is its relation to majorization order.  The following  simple calculation was observed in, e.g.,~\cite[Lem.~10]{DS81}:
\begin{lemma} \label{lem:eigenvalue-order}
If $\lambda \unrhd \mu$ then $c_\lambda > c_\mu$ and hence $d_\lambda < d_\mu$.
\end{lemma}

From this we may immediately determine the spectral gap\footnote{Although the spectral gap is usually notated $\lambda_1$, we avoid this notation here due to confusion with the standard notation~$\lambda$ for partitions.} of the transposition chain on $\symm{n}$, which is achieved at $\lambda = (n-1,1)$.

\begin{corollary} \label{cor:spectral-gap}
    The minimal nontrivial eigenvalue of operator $\Lap$ on~$V$ is $\frac{2}{n-1}$.
\end{corollary}
As we explain shortly, given $\lambda$ we will be particularly interested in the parameter $k = n - \lambda_1$.  An immediate consequence of \Cref{lem:eigenvalue-order} is that we can determine the minimal and maximal value of~$d_\lambda$ in terms of this parameter~$k$. We skip the straightforward calculations (most of which appear in~\cite[Ch.~3D, Lem.~2]{Dia88}):
\begin{corollary} \label{cor:extremal-eigenvalues}
For $\lambda \vdash n$ and $\lambda_1 = n-k$, we have
\[
    \frac{k}{n-1} \leq d_\lambda \leq \frac{2k}{n-1}.
\]
The upper bound has equality if $\lambda =  (n-k,1,\ldots,1)$.  Further, if $k \leq n/2$ we have
\[
    d_\lambda \geq \parens*{1 - \frac{k-1}{n}} \frac{2k}{n-1},
\]
with equality if $\lambda = (n-k,k)$.
\end{corollary}

Why consider the parameter $k = n - \lambda_1$? It turns out that this parameter is very much analogous to ``Fourier degree'' for functions on the Boolean cube, as the following result (proved in, e.g.,~\cite[Thm.~7]{EFP11}) shows:
\begin{theorem}                                     \label{thm:sn-degree}
    Let $f\colon \symm{n} \to \R$ be a nonzero function.  The \emph{degree} of~$f$ is the least $k \in \N$ such that $f$ can be represented as a linear combination of~``$k$-juntas'' (meaning functions $g$ such that $g(\pi)$ depends only on some $k$ values $\pi(j_1), \dots, \pi(j_k)$).  It is also equal to the least~$k$ such that $f^{=\lambda} = 0$ for all $\lambda$ with $n - \lambda_1 > k$.
\end{theorem}

We now provide two simple applications of \Cref{cor:extremal-eigenvalues} concerning functions of bounded degree:
\begin{lemma} \label{lem:total-influence-ub}
    If $f\colon \symm{n} \to \R$ has degree at most~$k$, then $\Inf[f] \leq kn \|f\|_2^2$.
\end{lemma}
\begin{proof}
Using \Cref{eq:frobenius-lap} and \Cref{cor:extremal-eigenvalues},
\begin{equation*}
\Inf[f] = \tbinom{n}{2} \langle f,\Lap f \rangle =  \tbinom{n}{2} \sum_{\lambda \unrhd (n-k, 1, \dots, 1)} d_\lambda \norm{f^{=\lb}}_2^2 \leq \sum_{\lambda \unrhd \hamlevel} kn \norm{f^{=\lb}}_2^2 = kn\norm{f}_2^2. \qedhere
\end{equation*}
\end{proof}
\begin{lemma} \label{lem:Ht-bound}
If $f\colon \symm{n} \to \R$ has degree at most~$k$, then for all $t \geq 0$,
\[
 \|\Heat_tf\|_2 \geq e^{-2kt/(n-1)} \|f\|_2,
\]
and for all $t \leq 0$,
\[
 \|\Heat_tf\|_2 \leq e^{-2kt/(n-1)} \|f\|_2.
\]
\end{lemma}
\begin{proof}
\Cref{cor:extremal-eigenvalues} shows that for $t \geq 0$,
\[
 \|\Heat_tf\|_2^2 = \sum_{\lambda \unrhd (n-k,1,\ldots,1)} e^{-2td_\lambda} \|f^{=\lambda}\|_2^2 \geq e^{-4kt/(n-1)} \|f\|_2^2.
\]
The bound for $t \leq 0$ follows in a similar fashion.
\end{proof}

We now move on to discussing functions on the multislice.  Let $\kappa \in \N_+^\ell$ be a histogram of size~$n$.  By relabeling the colors we may assume $\kappa_1 \geq \kappa_2 \geq \cdots \geq \kappa_\ell$ and hence that $\kappa \vdash n$. Let us denote by $u_0 \in [\ell]^n$ the following \emph{canonical string}:
\[
    u_0 = \underbrace{11\cdots 1}_{\substack{\kappa_1 \\ \text{times}}}\underbrace{22\cdots 2}_{\substack{\kappa_2 \\ \text{times}}} \cdots \underbrace{\ell \ell \cdots \ell}_{\substack{\kappa_\ell \\ \text{times}}}\,.
\]
Note that as $\pi$ runs over all permutations in $\symm{n}$, the string $u_0^\pi$ runs over all strings in $\slice\hamlevel$, with equal multiplicity $\kappa_1! \kappa_2! \cdots \kappa_\ell!$.  In this way, each function $f$ in the permutation module~$M^\kappa$ (i.e., the multislice $\slice\hamlevel$ considered as a representation of $\symm{n}$) can be naturally identified with a ``pullback'' function $\ol{f} \in V$, via $\ol{f}(\pi) = f(u_0^\pi)$.  Conversely, the functions $g \in V$ that correspond to functions on the multislice $\slice\hamlevel$ are precisely those that are invariant to the action of the \emph{Young subgroup} $\symm{\kappa_1} \times \cdots \times \symm{\kappa_\ell}$.
Classical results in the representation theory of the symmetric group show
that this subspace has the following isotypic decomposition:
\begin{equation}    \label{eqn:M-breakdown}
    M^\hamlevel = \bigoplus_{\lambda \unrhd \kappa} V^\lambda_\kappa,
\end{equation}
where $V^\lambda_\kappa$ is (isomorphic to) a nonzero subspace of $V^\lambda$ (specifically, $V^\lambda_\kappa$ consists of $K_{\lambda\kappa}$ copies of the irrep associated to~$\lambda$, where $K_{\lambda\kappa}$ is the \emph{Kostka number}.
Since this decomposition always includes $V^{(n-1,1)}_\kappa \leq V^{(n-1,1)}$ (unless $\kappa = (n)$), we conclude:
\begin{corollary} \label{cor:spectral-gap-multislice}
    The minimal nontrivial eigenvalue of the operator $\Lap$ on any~$M^\kappa$ (for $\kappa \neq (n)$) is also $\frac{2}{n-1}$.
\end{corollary}

We now define the notion of ``degree'' for functions on multislices:
\begin{definition}
    Let $f\colon \slice{\kappa} \to \R$ be a nonzero function.  The \emph{degree} of~$f$ is the least $k \in \N$ such that $f$ can be represented as a linear combination of~``$k$-juntas'' (functions $g$ such that $g(u)$ depends only on some $k$ values $u_{j_1}, \dots, u_{j_k}$).  It is also equal to the least~$k$ such that $f^{=\lambda} = 0$ for all $\lambda$ with $n - \lambda_1 > k$ (in $f$'s decomposition as in \Cref{eqn:M-breakdown}).
\end{definition}
\begin{claim}
    The two definitions of ``degree'' above are indeed the same.
\end{claim}
\begin{proof}
    If $g \in M^\kappa$ is a $k$-junta, it's easy to see that its pullback $\overline{g}\colon S_n \to \R$ is a $k$-junta.  Thus if $f \in M^\kappa$ is a linear combination of $k$-juntas, so too is its pullback $\overline{f}\colon S_n \to \R$.  From \Cref{thm:sn-degree} we get that $\overline{f}^{= \lambda} = 0$ for all $\lambda$ with $n - \lambda_1 > k$ and so the same is true of $f^{= \lambda}$.

    In the other direction, if $f^{= \lambda} = 0$ for all $\lambda$ with $n - \lambda_1 > k$, the same is true of $\overline{f}^{= \lambda}$, and hence $\overline{f}$ is a linear combination of $k$-juntas (by \Cref{thm:sn-degree} again).  We need to show that $f$ is also a linear combination of $k$-juntas.  By linearity, it suffices to assume that $\overline{f}$ is itself a $k$-junta; indeed, it further suffices to assume $\overline{f}$ is of the form $\overline{f}(\pi) = \ind[\pi(i_1) = j_1, \dots, \pi(i_k) = j_k]$ for some coordinates $i_1, \dots, i_k, j_1, \dots, j_k \in [n]$.  By definition we have $f(v) = \overline{f}(\pi)$ for any $\pi \in \symm{n}$ such that $u_0^\pi = v$.  In particular, it equals the average of $\overline{f}(\pi)$ over all such~$\pi$; i.e.,
    \[
        f(v) = \E_{\substack{\bpi \in \symm{n} \\ u_0^{\bpi} = v}}[\ol{f}(\bpi)] = \Pr_{\substack{\bpi \in \symm{n} \\ u_0^{\bpi} = v}}[\pi(i_1) = j_1, \dots, \pi(i_k) = j_k] = \ind[v_{i_1}=j_1,\dots,v_{i_k}=j_k].
    \]
    This means that $f$ is indeed a $k$-junta on~$\slice\kappa$.
\end{proof}
An immediate consequence is the following:
\begin{corollary}   \label{cor:its-okay}
    \Cref{lem:total-influence-ub,lem:Ht-bound} hold equally well for functions $f \in M^\hamlevel$ of degree at most~$k$.
\end{corollary}

Finally, we relate the main theorem in our paper to the comparison of norms for low-degree functions on the multislice:
\begin{lemma} \label{lem:hypercontractivity-bounded-degree}
Fix a histogram $\hamlevel \in \N_+^\ell$ and let $p = \min_i \hamlevel_i/n$.  Suppose that $f \in M^\hamlevel$ has degree~$k$.  Then for all finite $q \geq 2$:
\begin{align*}
 \|f\|_q &\leq (q-1)^{\Theta(k\log(1/p))} \|f\|_2, \\
 \|f\|_2 &\leq (q-1)^{\Theta(k\log(1/p))} \|f\|_{q'},
\end{align*}
where $q'$ is given by $1/q + 1/q' = 1$.
\end{lemma}
\begin{proof}
 \Cref{thm:main} shows that $\varrho_\hamlevel^{-1} = \Theta(n\log(1/p))$. \Cref{thm:hypercontractivity} therefore shows that $\|\Heat_tg\|_q \leq \|g\|_2$ and $\|\Heat_tg\|_2 \leq \|g\|_{q'}$ for all  $g \in M^\hamlevel$, where
\[
 t = \frac{\ln(q-1)}{2\varrho_{\hamlevel}} = \Theta(\ln(q-1) \cdot n \log(1/p)).
\]
Applying this to $g = \Heat_t^{-1} f$ (which has the same degree as~$f$) and using \Cref{lem:Ht-bound} (and \Cref{cor:its-okay}), we deduce
\[
 \|f\|_q \leq \|\Heat_{-t}f\|_2 \leq e^{2tk/(n-1)} \|f\|_2 = (q-1)^{\Theta(k\log (1/p))} \|f\|_2,
\]
and similarly for the second claimed inequality.
\end{proof}

We end this section by providing an analogue of Friedgut's Junta Theorem~\cite{Fri98} for functions on multislices:
\begin{theorem}\label{thm:friedgut-multislice}
Let $f \colon \slice\hamlevel \to \{0,1\}$  be such that $\Inf[f] \leq Kn$.  Write $p_i = \kappa_i/n$.  Then for every $\ep > 0$ there exists $h\colon \slice\hamlevel \to \{0,1\}$ depending on at most $\parens*{\frac{1}{p_1 p_2 \cdots p_\colors}}^{O(K/\eps)}$ coordinates such that \mbox{$\Pr_{\bu \sim \pi_\kappa}[f(\bu) \ne h(\bu)] \leq \ep$}.
\end{theorem}
The proof is essentially identical to Wimmer's proof~\cite[Sec.~VI]{Wim14} of the analogous theorem for functions on the Boolean slice (i.e., the $\ell = 2$ case of the above). After replacing Wimmer's pullback function (notated $f^g$ therein) with our generalization~$\overline{f}$, it only remains to substitute in our main log-Sobolev inequality for the multislice $\slice\hamlevel$.
\ignore{
The idea is that we will embed $g$ into some function $f^g\colon S_n\to \{-1,1\}$, in a way that preserves the action of transpositions. Then we will apply Theorem~V.6 of \cite{Wim14}, which gives a junta on $S_n$, and then finally show that we can project back to a junta on $\slice\hamlevel$.

We define the projection map $P_\hamlevel\colon S_n \to \slice\hamlevel$, such that $(P_\hamlevel(\sigma))_i = j$ if $\hamlevel_{j-1} \leq \sigma(i) < \hamlevel_j$ (take $\hamlevel_0 = 0$). $P_\hamlevel$ commutes with transpositions, i.e.\ $P_\hamlevel(\sigma\tau) = P_\hamlevel(\sigma)\tau$ for $\tau \in \transpos n$. Also, $P_\hamlevel$ composed with the uniform distribution over $S_n$ gives the uniform distribution over $\slice\hamlevel$.

Using this setup, we define $f^g\colon S_n \to \{-1,1\}$ by $f^g(\sigma) = g(P_\hamlevel\sigma)$.
Because $P_\hamlevel$ commutes with transpositions, $P_\hamlevel$ behaves well with influences and also with the Markov operator $\Heat_t$ (defined suitably on the symmetric group).
\begin{lemma}[Proposition VI.1 of \cite{Wim14}] \label{lem:projection-influences}
Let $g\colon \slice\hamlevel \to \R$ and let $f^g\colon S_n \to \R$ such that $f^g(\sigma) = g(P_\hamlevel \sigma)$, and let $\tau \in \transpos n$ and $t > 0$. Then
\begin{enumerate}[label={(\roman*)}]
\item $\Inf_\tau(f^g) = \Inf_\tau(g)$,
\item $f^{\Heat_t g} = \Heat_t f^g$,
\item $\Inf_\tau(\Heat_t f^g) = \Inf_\tau(\Heat_t g)$.
\end{enumerate}
\end{lemma}
We can characterize the functions on the symmetric group which are pullbacks of functions on the multislice. The proof is similar to that of Lemma~VI.2 in \cite{Wim14}.
\begin{lemma} \label{lem:projection_conditions}
    The following are equivalent:\rnote{change to widehats; also no need for range to be Boolean?}
    \begin{enumerate}
    \item $f\colon S_n \to \{-1,1\}$ satisfies $\hat f_\lambda = \rho_\lambda((i\ j)) \hat f_\lambda$ for all $\lambda \vdash n$ for all $1 \leq k \leq \colors$ and $\hamlevel_{k-1} < i < j \leq \hamlevel_{k}$ (again taking $\hamlevel_0 = 0$).
    \item There exists $g\colon \slice\hamlevel \to \{-1,1\}$ such that $f=f^g$.
    \end{enumerate}
\end{lemma}

\begin{proof}[Proof of \Cref{thm:friedgut-multislice}]
We can view the multislice as the Schreier graph $\operatorname{Sch}(G,X,U)$ with the group $G = S_n$ acting on $X = \slice\hamlevel$, and the set of generators $U = \transpos{n}$.

Let $\tau \in \transpos n$. Since the set of transpositions is a conjugacy class in $S_n$, \cite[Prop.~2.5]{OW13a} says that $\Heat_t$ and $L_\tau$ commute for all $t$. We use this and hypercontractivity (\Cref{thm:hypercontractivity} with $q = 4$) to get that for $t = \ln 3/(2\varrho_\hamlevel)  = O(n \log \frac{1}{\delta})$,
\[
    \Inf_\tau(\Heat_tg) = \norm{L_\tau \Heat_t g}_2^2 = \norm{\Heat_tL_\tau g}_2^2 \leq \norm{L_\tau g}_{4/3}^2 = \E[|L_\tau g|^{4/3}]^{3/2}.
\]
Since $|L_\tau g|$ only takes on the two values $0$ and $1$, we can scale this as $|L_\tau g|^{4/3} = |L_\tau g|^2$, so
\[
    \E[|L_\tau g|^{4/3}]^{3/2} = (\norm{L_\tau g}_2^2)^{3/2} = \Inf_\tau (g)^{3/2}.
\]
Using \Cref{lem:projection-influences},
\[
    \Inf_\tau(\Heat_tf^g) = \Inf_\tau(\Heat_t g) \leq \Inf_\tau (g)^{3/2} = \Inf_\tau(f^g)^{3/2}.
\]
This then satisfies the conditions for Theorem~V.6 of \cite{Wim14}, so there exists a function $h\colon S_n \to \{-1,1\}$ depending on $2^{O_\delta(k/\ep)}$ coordinates such that $\Pr[g(\bx)\ne h(\bx)] \leq \ep$. Since $h$ given by Theorem~V.6 satisfies that for each $\lambda \vdash n$, $\hat h_\lambda$ is either $\hat f^g_\lambda$ or 0, this fulfills \Cref{lem:projection_conditions} so $h = f^{g'}$ for some $g'\colon \slice\hamlevel \to \{-1,1\}$, and $\Pr[g(\bx)\ne g'(\bx)] \leq \ep$.
\end{proof}
}

\subsection{Nisan--Szegedy Theorem on the multislice}
The Nisan--Szegedy Theorem says that a degree-$k$ Boolean-valued function on the Hamming cube is a $k2^k$-junta.  (We remark that the smallest quantity~$\gamma_2(k)$ that can replace $k2^k$ here is now known~\cite{CHS18} to satisfy $3 \cdot 2^{k-1}-2 \leq \gamma_2(k) < 22 \cdot 2^k$.)  In~\cite{FI18b}, an analogous result for functions on Hamming slices was shown; they conjectured a similar result for functions on multislices.  We resolve this conjecture, following the structure of their proof.  This proof structure involves proving three successively stronger versions of the desired theorem.

The first version pertains only to functions on balanced multislices; it was originally established for Hamming slices in~\cite{FKMW18}:
\begin{theorem}                                     \label{thm:ns1}
    Fix $\colors \geq 2$, assume $n$ is a multiple of~$\colors$, and let $\hamlevel = (n/\colors, \dots, n/\colors) \in \N_+^\colors$.  If $f\colon \slice\hamlevel \to \{0,1\}$ has degree at most~$k$, then $f$ is an $\colors^{O(k)}$-junta.
\end{theorem}
To prove \Cref{thm:ns1}, we first use our hypercontractivity result to establish the following analogue of \cite[Lem.~3.1]{FI18b}:
\begin{lemma} \label{lem:influence-lb}
    In the setting of \Cref{thm:ns1}, every nonzero influence $\Inf_\tau[f]$ is at least $\colors^{-O(k)}$ (where the $O(\cdot)$ hides a universal constant).
\end{lemma}
\begin{proof}
    Since $f$ has degree at most~$k$, the same is true of $\Lap_\tau f$.  Thus \Cref{lem:hypercontractivity-bounded-degree} shows (taking, say, $q = 4$) that
    \[
        \Inf_\tau[f] = \tfrac12 \|\Lap_\tau f\|_2^2 \leq \colors^{O(k)} \cdot \|\Lap_\tau f\|^2_{4/3}.
    \]
    Since $f$ is Boolean-valued, $\Lap_\tau f$ takes values in $\{0,\pm 1\}$, and so
    \[
        \|\Lap_\tau f\|^2_{4/3} = (\|\Lap_\tau f\|_2^2)^{3/2} = (2\Inf_\tau[f])^{3/2}.
    \]
    Combining these yields that either $\Inf_\tau[f] = 0$ or else $\Inf_\tau[f]^{-1/2} \leq \colors^{O(k)}$, as needed.
\end{proof}
\noindent With \Cref{lem:influence-lb} in hand (as well as \Cref{lem:total-influence-ub}) follows exactly as in~\cite[Sec.~3.1]{FI18b}.\\

The same argument works as long as $\min_i \{\hamlevel_i\} = \Omega(n)$.
The second, stronger version of Nisan--Szegedy for multislices shows that in fact, it suffices to assume only that $\min_i \{\hamlevel_i\} \geq \ell^{O(k)}$.
\begin{theorem}                                     \label{thm:ns2}
    There are universal constants~$C \geq C'$ such that the following holds.  For all $k \in \N_+$ and all $\hamlevel \in \N_+^\colors$ with $\min_i \{\hamlevel_i\} \geq \colors^{C k}$, if $f\colon \slice\hamlevel \to \{0,1\}$ has degree at most~$k$, then $f$ is an $\colors^{C' k}$-junta.
\end{theorem}
Roughly speaking, $C = 3C'$, where $C'$ is the constant hidden in the $O(\cdot)$ of \Cref{thm:ns1}.  The proof of \Cref{thm:ns2} exactly follows the argumentation in~\cite[Sec.~3.2]{FI18b}.  Essentially, starting from \Cref{thm:ns1},  they show that the truth of the statement is preserved whenever one of the quantities~$\hamlevel_i$ is incremented.\\

Before stating our third Nisan--Szegedy variant, let us extend the definition of $\gamma_2(k)$; we'll define $\gamma_\colors(k)$ to be the least integer such that the following statement is true:
\[
    \text{Every degree-$k$ Boolean-valued function $f\colon [\colors]^n \to \{0,1\}$ on the ``$\ell$-multicube'' is a $\gamma_\colors(k)$-junta.}
\]
Here we say that $f\colon [\colors]^n \to \R$ has degree at most~$k$ if it is a linear combination of $k$-juntas (as usual for functions on product spaces, see \cite[Def.~8.32]{OD14}).  As remarked at the end of~\cite{FI18b}, it's easy to show that $\colors^{k-1} \leq \gamma_\colors(k) \leq \gamma_2(\lceil \log_2 \colors \rceil k)$.  When $\colors$ is a power of~$2$, this upper bound is at most $22 \cdot \colors^k$, very close to the lower bound; in general we have $\gamma_\colors(k) < 22 \cdot \colors^{2k}$.

Our third and final Nisan--Szegedy Theorem for the multislice improves the junta size in \Cref{thm:ns2} to~$\gamma_\colors(k)$, which is optimal (since the analogue of~\cite[Lem.~3.10]{FI18b} equally holds in our setting).  We do not know, however, the weakest lower bound we can assume on $\min_i\{\hamlevel_i\}$.
\begin{theorem}                                     \label{thm:ns3}
    There is a universal constant~$C$ such that the following holds.  For all $k \in \N_+$ and all $\hamlevel \in \N_+^\colors$ with $\min_i \{\hamlevel_i\} \geq \colors^{C k}$, if $f\colon \slice\hamlevel \to \{0,1\}$ has degree at most~$k$, then $f$ is an $\gamma_\colors(k)$-junta.
\end{theorem}
The way we prove this departs somewhat from the polynomials-based proof in~\cite[Lem.~3.9]{FI18b}.
\begin{proof}
From \Cref{thm:ns2} we know that $f$ is an $\colors^{C' k}$-junta, where $C' \leq C$. Without loss of generality, say that $f(u)$ depends only on coordinates $u_1, \dots, u_J$, where $J \leq \colors^{C' k}$. Note that as we vary $u \in \slice\hamlevel$, we see all $\colors^{J}$ possibilities for the substring $(u_1, \dots, u_J)$; this is because $\min_i \{\hamlevel_i\} \geq J$.  As a consequence, we can define a function $g\colon [\colors]^J \to \{0,1\}$ by
\[
 g(u_1, \dots, u_J) = \E_{\substack{u_{J+1},\ldots,u_n \in [\colors]\colon\\ u \in \slice\hamlevel}}[f(u)].
\]
Since $f$ has degree $k$, it is a linear combination of $k$-juntas, and in particular, a linear combination of functions of the form $\ind[u_{i_1}=c_1,\dots,u_{i_k}=c_k]$. We show below that
\[
h_{i,c} := \E_{\substack{u_{J+1},\ldots,u_n \in [\colors]\colon\\ u \in \slice\hamlevel}}\bigl[\ind[u_{i_1}=c_1,\dots,u_{i_k}=c_k]\bigr]
\]
is a degree $k$ function, and so $g$ has degree $k$. Thus $g$ is (by definition) a $\gamma_\colors(k)$-junta, and hence so is $f$.

It remains to show that $h_{i,c}$ has degree $k$. Suppose first that $i_1,\dots,i_k > L$. Let $d = (\#_1c,\dots,\#_\colors c)$ be the histogram of $c_1,\dots,c_k$, and let $w = (\#_1u_{\leq L},\dots,\#_\colors u_{\leq L})$ be the histogram of $u_1,\dots,u_L$. The reader can verify that
\[
h_{i,c} = \frac{(\hamlevel_1-w_1)^{\underline{d_1}}\dots(\hamlevel_\colors-w_\colors)^{\underline{d_\colors}}}{(n-L)^{\underline{k}}},
\]
where $a^{\underline{b}} = a(a-1)\dots(a-b+1)$. Since $d_1+\dots+d_\colors=k$, this is a degree $k$ function.

When some of the indices $i_1,\dots,i_k$ are in $[L]$, we have to modify the argument slightly. Suppose that $i_1,\dots,i_r \in [L]$ and $i_{r+1},\dots,i_k \notin [L]$. Redefine $d$ to capture the histogram of $c_{r+1},\dots,c_k$. The reader can verify that
\[
h_{i,c} = \ind[u_{i_1}=c_1,\dots,u_{i_r}=c_r] \times \frac{(\hamlevel_1-w_1)^{\underline{d_1}}\dots(\hamlevel_\colors-w_\colors)^{\underline{d_\colors}}}{(n-L)^{\underline{k-r}}},
\]
which has degree $k$ since $d_1+\dots+d_\colors=k-r$.
\end{proof}

\bibliographystyle{alpha}
\bibliography{odonnell-bib}

\end{document}